\documentclass[12pt]{article}

\usepackage{amsthm,amsfonts,amsmath, amscd}
\usepackage{verbatim} 
\usepackage{color}
\usepackage{mathtools} 
\usepackage{booktabs} 
\usepackage{enumitem} 

\newcommand{\makeblue}[1]{#1} 
\newcommand{\makered}[1]{#1} 
\newcommand{\makeviolet}[1]{#1} 
\newcommand{\modif}[1]{#1} 
                   
\theoremstyle{plain}
\newtheorem{thm}{Theorem}[section]
\newtheorem{lm}[thm]{Lemma}
\newtheorem{cl}[thm]{Corollary}
\newtheorem{prop}[thm]{Proposition}
\newtheorem{ex}[thm]{Example}

\theoremstyle{definition}
\newtheorem{defi}[thm]{Definition}
\newtheorem{remark}[thm]{Remark}
\theoremstyle{remark}

\theoremstyle{plain}

\numberwithin{equation}{section}

\newcommand{\R}{{\mathord{\mathbb R}}}
\newcommand{\C}{{\mathord{\mathbb C}}}

\newcommand{\N}{{\mathord{\mathbb N}}}
\newcommand{\Cn}{\C^n}

\newcommand{\Cnn}{\C^{n\times n}}

\DeclareMathOperator{\diag}{diag}

\DeclareMathOperator{\range}{range}
\DeclareMathOperator{\rank}{rank}
\DeclareMathOperator{\toeplitz}{toeplitz}

\newcommand{\BH}{B_H} 
\newcommand{\BA}{B_A} 

\newcommand{\mB}{B} 

\newcommand{\cO}{{\bigO}}

\newcommand{\bigO}{{\mathcal{O}}}

\newcommand{\be}{{\mathbf e}}

\newcommand{\bv}{{\mathbf v}}

\newcommand{\bx}{{\mathbf x}}

\newcommand{\dd}[1]{\,\operatorname{d}\!#1}

\newcommand{\ddt}{\frac{\dd}{\dd{t}}}
\newcommand{\ddtau}{\frac{\dd}{\dd{\tau}}}

\newcommand{\ip}[2]{\langle {#1}, {#2} \rangle}

\newcommand{\norm}[1]{\| {#1} \|}

\newcommand{\eps}{\epsilon}

\newcommand{\mHC}{{m_{HC}}}

\newcommand{\setC}{\mathcal{C}}
\newcommand{\sphere}{\mathcal{S}}

\newcommand{\tE}{\widetilde{E}}
\newcommand{\tm}{\widetilde{m}}

\newcommand{\ttilde}[1]{\widetilde{\raisebox{0pt}[0.85\height]{$\widetilde{#1}$}}}
\newcommand{\ttT}{\ttilde{T}}
\newcommand{\tttT}{\widetilde{\widetilde{T}}}

\usepackage{stix}
\usepackage{hyperref}
\hypersetup{pdfauthor=author}

\begin{document}
\title{{\sc  The Hypocoercivity Index for the short time behavior of linear time-invariant ODE systems}}

\author{Franz Achleitner\thanks{TU Wien, Institute of Analysis and Scientific Computing, Wiedner Hauptstr. 8-10, A-1040 Wien, Austria, franz.achleitner@tuwien.ac.at},
Anton Arnold\thanks{TU Wien, Institute of Analysis and Scientific Computing, Wiedner Hauptstr. 8-10, A-1040 Wien, Austria, {anton.arnold@tuwien.ac.at}}, and 
Eric A.\ Carlen\thanks{Department of Mathematics, Rutgers University, 110 Frelinghuysen Rd., Piscataway NJ 08854, USA, carlen@math.rutgers.edu}
} 

\maketitle

\begin{abstract}
We consider the class of conservative-dissipative ODE systems, which is a subclass of Lyapunov stable, linear time-invariant ODE systems. 
We characterize asymptotically stable, conservative-dissipative ODE systems via the hypocoercivity (theory) of their system matrices. 
Our main result is a concise characterization of the hypocoercivity index (an algebraic structural property of matrices with positive semi-definite Hermitian part introduced in~\textit{Achleitner, Arnold, and Carlen (2018)}) in terms of the short time behavior of the \makeblue{norm of the matrix exponential} for the associated conservative-dissipative ODE system.
\end{abstract}

\medskip
\centerline{Keywords: semi-dissipative ODE systems, hypocoercivity (index)}

\centerline{MSC: 34D05, 34A30, 34Exx}

\maketitle

\section{Introduction}\label{sec:intro}
In this paper we shall use hypocoercivity techniques to characterize the short time behavior of linear time-invariant ODE systems of the form 
\begin{equation}\label{ODE:B} 
  \bx'(t) = -B\bx(t)\,, 
\end{equation}
with matrices $B\in\Cnn$ whose Hermitian part $\BH :=(B+B^*)/2$ is positive semi-definite\footnote{We use the following notation:
The conjugate transpose of a matrix $B\in\Cnn$ is denoted by $B^*$.
Positive definiteness (resp.~semi-definiteness) of Hermitian matrices is denoted by $B>0$ (resp.~$B\geq 0$).}, such that~$-B$ is a conservative-dissipative (or semi-dissipative) matrix, see Definition~\ref{def:con-dis} below.
An extension to~\emph{stable systems} of the form~\eqref{ODE:B}, i.e.\ matrices~$B$ having their spectrum in the closed right half plane where, additionally, purely imaginary eigenvalues are non-defective\footnote{An eigenvalue is non-defective if its algebraic and geometric multiplicities coincide.}, are discussed in the parallel article~\cite{AAM22}. 

Concerning the short time behavior \makeblue{of the propagator~$P(t)$ of~\eqref{ODE:B}, given by the fundamental matrix~$P(t):=e^{-Bt}$, we shall be interested in estimates on the spectral norm of the matrix exponential $e^{-Bt}$} of the form 
\begin{equation}\label{short-t-decay1}
  \|P(t)\|_2 = 1-c\,t^a+\bigO(t^{a+1})\quad\mbox{ for } t\to 0^+\,,
\end{equation}
where $c>0$ and $a\in\N$. 
\makeblue{For practical reason, we shall refer to the spectral norm of the matrix exponential $e^{-Bt}$ as the \emph{propagator norm}.}
The large time behavior, including conditions guaranteeing exponential decay estimates of the form
\begin{equation}\label{HC-decay}
  \|P(t)\|_2 \le c\,e^{-\mu t},\qquad t\ge0\,,
\end{equation}
with some $c\ge1$, and $\mu>0$ will be discussed in the  forthcoming article \cite{AAC22}. 

Following \cite{Vi09}, the matrix $B$ is called \emph{hypocoercive} if~\eqref{HC-decay} holds, and $B$ is \emph{coercive} if and only if $\|P(t)\|_2 \le e^{-\mu t}$ for some $\mu>0$ and all $t\ge0$. 
For the intermediate time regime we shall be interested in a \emph{waiting time} $t_0$ such that all solutions will have decayed (at least) by a factor $1/e$, i.e.
\begin{equation}\label{propagator_norm:waiting_time}
  \|P(t)\|_2 \le \frac1e\,, \qquad t\ge t_0\,.
\end{equation}
Here we shall include numerical evidence on this intermediate time regime; a detailed analysis will be provided in \cite{AAC22}.

Our interest in systems~\eqref{ODE:B} is motivated by non-equilibrium statistical physics and, in particular, kinetic theory. 
There, a frequently encountered class of linear equations 
has the form
\[
 \frac{\partial}{\partial t}f(t,x,v)
 = -v\cdot \nabla_x f(t,x,v) + Q f(t,x,v)\ ,
\]
where $Q$ is a linear operator, the so-called {\em linearized collision operator}, describing changes in velocity resulting from binary collisions.
Take the domain of the $x$ variable to be a $d$--dimensional torus $\mathcal{T}$ of side-length $L$.   
The operator $v\cdot \nabla_x$, called the streaming operator is anti-Hermitian on a weighted $L^2$-space on $\mathcal{T}\times \R^d$, whereas $Q$ is negative semi-definite on the same weighted $L^2$-space.
Concerning examples we refer to \cite[\S1.4]{DoMoSch15}, \cite{AAC16, AAC18} where a modal decomposition (in $x$) of kinetic BGK-type equations\footnote{named after the physicists Bhatnagar--Gross--Krook.} led to ODE systems like~\eqref{ODE:B}; in the case of continuous velocities it actually led to ``infinite matrices'' $B$. 
In the follow-up paper \cite{ASS20} a spectral decomposition of Fokker--Planck equations allowed for a precise analysis of their short time behavior in the spirit of \eqref{short-t-decay1}.

Very often, one is interested in initial data $f_0(x,v)$ for the equation that are nearly constant on spatial scales much smaller than $L$, where $L$ is large, so that $v\cdot \nabla_x f(t,x,v)$ will be of order $1/L$. 
Finite dimensional approximations of such systems result in systems of ODEs of the form \eqref{ODE:B}, where its Hermitian part $\BH$ is positive semi-definite.
Hence it is often natural to write~\eqref{ODE:B} in the form
\begin{equation}\label{ODE:epsA+C}
 \bx'(t) = -\epsilon A\bx(t) -C\bx(t)\,,
\end{equation}
where $\epsilon\in\R$, $A\in\Cnn$ is anti-Hermitian, and $C\in\Cnn$ is positive semi-definite Hermitian. 
The systems~\eqref{ODE:B} and~\eqref{ODE:epsA+C} are related via $\BA=\epsilon A$ and $\BH=C$, where $\BA :=(B-B^*)/2$ denotes the anti-Hermitian part of~$B$.
One is then led to study the asymptotics of the solution as $\epsilon\to 0$.
This is often referred to as \emph{asymptotic limit}.

\medskip
The main goal of this paper is to give a concise interpretation of the \emph{hypocoercivity index} (an algebraic structural property of the matrix $B$, introduced in \cite{AAC18}) in terms of the exponent $a$ in short time estimate~\eqref{short-t-decay1}. 

The rest of this paper is organized as follows: 
In the remainder of \S\ref{sec:intro} we define and characterize \emph{conservative-dissipative ODE systems} and \emph{hypocoercive matrices}. 
In \S\ref{subsec:HC-index} we give three equivalent definitions of the \emph{hypocoercivity index} of a matrix and three corresponding \emph{Kalman rank conditions}. 
Our main result, Theorem \ref{th:HC-decay}, gives a sharp characterization of the hypocoercivity index of a matrix $B$ in terms of the short time decay of 
\makeblue{the spectral norm of the associated matrix exponential~$e^{-Bt}$}.
It is presented in \S\ref{sec:HC-interpr}, while technical parts of the proof are deferred to \ref{app:short-time-decay}. 
Finally, \S\ref{subsect:num} gives a numerical illustration of Theorem \ref{th:HC-decay}.



%
\subsection{Conservative-dissipative systems of ODEs}
\label{sec11}
It is well known that the null solution~$\bx(t)\equiv 0$ of a linear system~\eqref{ODE:B} is \emph{(Lyapunov) stable} if all eigenvalues of~$-B$ have non-positive real part and the eigenvalues on the imaginary axis are non-defective, and the null solution~$\bx(t)\equiv 0$ of~\eqref{ODE:B} is \emph{asymptotically stable} if all eigenvalues of~$-B$ have negative real part.
For practical reasons, if the null solution~$\bx(t)\equiv 0$ of a linear system~\eqref{ODE:B} is (asymptotically) stable then we will call~\eqref{ODE:B} an (asymptotically) stable system.

Consider a linear system of ODEs~\eqref{ODE:B} with matrix $B\in\Cnn$.
Then the derivative of the squared Euclidean norm of a solution $\bx(t)$ satisfies
\begin{equation} \label{eq:energy}
 \ddt \|\bx(t)\|_2^2 
 = \langle -B\bx(t),\bx(t)\rangle +\langle \bx(t),-B\bx(t)\rangle
 = -2 \langle \bx(t),\BH\bx(t)\rangle .
\end{equation}
Therefore a sufficient condition for $B$ to generate a stable system~\eqref{ODE:B} is that its Hermitian part is positive semi-definite\footnote{However, the Hermitian part of a matrix~$B$ pertaining to a stable system does not have to be positive semi-definite; a different equivalent norm will generally yield a different sufficient condition.}.
This fact and the importance of this subclass of stable systems in kinetic theory inspires the following definition:
\begin{defi} \label{def:con-dis}
A matrix $-B\in\Cnn$ is called \emph{dissipative} (resp. \emph{conservative-dissipative} or \emph{semi-dissipative}) if the Hermitian part of~$-B$ is negative definite (resp. negative semi-definite). 

For a (conservative-)dissipative matrix $-B\in\Cnn$, the associated 
system of ODEs~\eqref{ODE:B} is called a \emph{(conservative-)dissipative ODE system}.

For practical reasons, a matrix $B\in\Cnn$ is called \emph{positive conservative-dissipative} 
if the Hermitian part of~$B$ is positive semi-definite. 
\end{defi} 
This conservative-dissipative property of a matrix~$B$ is invariant under unitary transformations, but it is not invariant under similarity transformations (see~\cite{AAM22} for details).
Let $\lambda^{\BH}_{min}$ and $\lambda^{\BH}_{max}$ denote the least and the greatest eigenvalues of $\BH$, respectively. 
By the Rayleigh-Ritz variational principle
\[
\lambda^{\BH}_{min}
 =\min_{\|\bx\|_2=1} \ip{ \bx}{ \BH\bx} \qquad{\rm and}\qquad 
\lambda^{\BH}_{max}
 =\max_{\|\bx\|_2=1} \ip{ \bx}{ \BH\bx}\ .
\]
It follows immediately that
\begin{equation}\label{spectral-incl} 
   \lambda^{\BH}_{min}
    \leq \min\{\Re \lambda\,:\, \lambda\in\sigma(B)\}
    \leq \max\{\Re \lambda\,:\, \lambda\in\sigma(B)\}
    \leq \lambda^{\BH}_{max} \,.
\end{equation}
%
%

%
\subsection{Hypocoercive matrices.}
In the introduction, matrices $B\in\Cnn$ are called \emph{hypocoercive} if the associated propagator $P(t):=e^{-tB}$ satisfies~\eqref{HC-decay} for some $c\geq 1$ and $\mu>0$.
In other words, $B$ is hypocoercive if (the null-solution $\bx(t)\equiv 0$ of)
the associated linear system~\eqref{ODE:B} is exponentially stable.
For linear time-invariant systems~\eqref{ODE:B}, exponential stability is equivalent to asymptotic stability (i.e. all solutions approach the origin in the large time limit), which is equivalent to the condition that all eigenvalues of $-B$ have negative real part.
Starting with the classical notion of a coercive operator/matrix, we characterize hypocoercive matrices as follows (see e.g. \cite{AAC18}):
\begin{defi} \label{def:hypocoercive}
A matrix $B\in\Cnn$ is called \emph{coercive} if its Hermitian part~$\BH$ is positive definite, and it is called \emph{hypocoercive} if the spectrum of~$B$ lies in the \emph{open} right half plane.
\end{defi}
Hypocoercive matrices are often called \emph{positive stable}. 
We use the notion of hypocoercivity to emphasize the link to the analogous situation in partial differential equations, see~\cite{AAC18,ArEr14,Vi09}.

\medskip
It is well-known, see \eqref{spectral-incl}, that for positive conservative-dissipative matrices~$B\in\Cnn$, the spectrum of~$B$ lies in the closed right half plane, but there may be purely imaginary eigenvalues.
In this case, the existence of purely imaginary eigenvalues can be characterized as follows:
\begin{prop}[{\cite[Lemma 3.1]{MMS16}, \makeblue{\cite[Lemma 2.4 with Proposition~1(B2), (B4)]{AAC18}}}] \label{prop:border}
Let $\mB\in\Cnn$ be (positive) conservative-dissipative.
Then, $\mB$ has an eigenvalue on the imaginary axis if and only if $\BH v =0$ for some eigenvector~$v$ of~$\BA$. 
\end{prop}
Hence, a positive conservative-dissipative matrix~$\mB$ is hypocoercive if and only if no eigenvector of the anti-Hermitian part lies in the kernel of the Hermitian part. 
The latter condition is well known in control theory, and there exists a range of equivalent characterizations, see e.g.~\cite[Proposition 1]{AAC18} and~\cite[Lemma 3.1]{MMS16}:
For example, positive conservative-dissipative matrices~$B\in\Cnn$ are hypocoercive if and only if  
\begin{equation}\label{HC-charact}
 \mbox{``No non-trivial subspace of $\ker(\BH)$ is invariant under $\BA$.'' }
\end{equation} 
The characterization~\eqref{HC-charact} implies the following result:
\begin{lm}\label{lm:HC+eps}
Let $B\in\Cnn$ be positive conservative-dissipative.
Then $B =\BA+\BH$ is hypocoercive if and only if $\epsilon \BA+\BH$ is hypocoercive for all $\epsilon\neq 0$.
\end{lm}

But even if hypocoercivity of~$B$ is known, it is not trivial to obtain an exponential decay estimate~\eqref{HC-decay} with a quantitative (or even optimal) decay rate~$\mu$. 
Indeed, a simple energy estimate (i.e.\ pre-multiplying~\eqref{ODE:B} by $\bx^*$) or using Trotter's product formula only yields conservative-dissipativity of the system, but no decay: 
So, its propagator $P(t) =e^{-B t}$ satisfies at least $\|P(t)\|_2 \le1$ for all $t\ge0$. 

Let us also comment on the short time behavior. 
If a decay estimate~\eqref{short-t-decay1} holds with some exponent $a>1$, then an estimate of the form $\|P(t)\|_2 \le e^{-\mu t}$ (as it is typical for coercive matrices~$B$) is impossible (consider the Taylor expansion of $\|P(t)\|_2$ around $t=0$). 
In such cases the system can only be hypocoercive, along with an estimate~\eqref{HC-decay} with $c>1$.

Our conditions for hypocoercivity lead naturally to the notion of an {\em index of hypocoercivity}, which is a non-negative integer whenever $B$ is hypocoercive, and is zero if and only if $B$ is coercive (see~\S\ref{subsec:HC-index}).

%
%
%
%
\section{Hypocoercivity index and the short time decay of conservative-dissipative ODE systems}
\label{sec:HC-index}
In this section we shall present our main result, i.e.\ a concise characterization of the hypocoercivity index in terms of the short time behavior of conservative-dissipative ODE systems~\eqref{ODE:B}.

\subsection{Hypocoercivity index}
\label{subsec:HC-index}
First we recall from \cite[\S2.2]{AAC18} the definition of the hypocoercivity index:
\begin{defi}\label{def-HCIndex-full}
Let $B\in\Cnn$ be positive conservative-dissipative. 
Its \emph{hypocoercivity index (HC-index)}~$\mHC$ is defined as the smallest integer $m\in\N_0$ (if it exists) such that the matrix
\begin{equation} \label{HCI:Tm}
 T_m 
 :=\sum_{j=0}^{m} \BA^j \BH(\BA^*)^j
\end{equation} 
is positive definite. 
The matrix~$B$ is coercive iff $\mHC=0$, hypocoercive iff $\mHC\in\N_0$ (due to \makeblue{\cite[Lemma 2.4]{AAC18}}), 
and for non-hypocoercive matrices~$B$ we set $\mHC=\infty$. 
\end{defi}
The Hermitian matrix~$T_m$ in Definition~\ref{HCI:Tm} readily shows that the hypocoercivity index of a positive conservative-dissipative matrix~$B$ is invariant under unitary congruence transformations but, in general, not under similarity transformations (see~\cite{AAM21} for details).

In \cite[Lemma 2.3]{AAS15} it was proven that this index equals the smallest integer $m\in\N_0$ such that 
\begin{equation}\label{kalman1}
  \rank \big\{ \sqrt{\BH},\,\BA\sqrt{\BH},\ldots,\BA^{m}\sqrt{\BH} \big\}=n\,,
\end{equation}
which is often called {\em Kalman rank condition} (see also~\cite[{Proposition 1(B1)}]{AAC18}).
\begin{remark}\label{remark:HC+eps}
When considering rather $\eps \BA+\BH$, its hypocoercivity property and its index of hypocoercivity are independent of $\eps\ne0$, which follows trivially from~\eqref{kalman1}. 
Due to the above mentioned equality of the indices, the positive definiteness of $\sum_{j=0}^{\mHC} (\eps \BA)^j \BH(\eps \BA^*)^j$ is hence also independent of $\eps\ne0$.
\end{remark}
Next we shall present four (equivalent) variants of the Kalman rank condition:
\begin{lm}\label{lem:Kalman}
 Let $B\in\Cnn$ be positive conservative-dissipative.
 Consider the following four Kalman rank conditions:
\begin{subequations} \label{KRC} 
 \begin{align}
\exists m\in\N_0: &\quad\rank\{\sqrt{\BH},\ \BA\sqrt{\BH},\ \ldots,\ \BA^m\sqrt{\BH}\}=n\ , \label{KRC:a} 
\\
\exists m\in\N_0: &\quad\rank\{\sqrt{\BH},\ B\sqrt{\BH},\ \ldots,\ B^m\sqrt{\BH}\}=n \ , \label{KRC:b} 
\\ 
\exists m\in\N_0: &\quad\rank\{\sqrt{\BH},\ B^*\sqrt{\BH},\ \ldots,\ (B^*)^m\sqrt{\BH}\}=n \ , \label{KRC:b'}
\\ 
\exists m\in\N_0: &\quad\rank\{C_0,C_1,\ \ldots,\ C_m\}=n \text{ with } C_0:=\sqrt{\BH};\ C_{j+1}:=[C_j,\BA]\makeblue{=C_j B_A -B_A C_j},\ j\in\N_0\ . \label{KRC:c}
 \end{align}
The conditions~\eqref{KRC:a}--\eqref{KRC:c} are equivalent in the sense that, if there exists $m\in\N_0$ such that one condition holds, then the other three conditions hold as well for the same $m$. 
\end{subequations}
\end{lm}

\begin{proof}
In fact, we prove that for all $m\in\N_0$ the ranges of all four matrices in~\eqref{KRC} are equal,
\begin{subequations}\label{KRC:range}
 \begin{align}
 \range\{\sqrt{\BH},\ \BA\sqrt{\BH},\ \ldots,\ \BA^m\sqrt{\BH}\} 
&=\range\{\sqrt{\BH},\ B\sqrt{\BH},\ \ldots,\ B^m\sqrt{\BH}\} \label{KRC:range:a} 
\\
&=\range\{\sqrt{\BH},\ B^*\sqrt{\BH},\ \ldots,\ (B^*)^m\sqrt{\BH}\} \label{KRC:range:a'} 
\\
&=\range\{C_0,C_1,\ \ldots,\ C_m\} , \label{KRC:range:b}
 \end{align}
\end{subequations}
which can be done inductively:
 
The base case $m=0$ is trivial, since all four matrices in~\eqref{KRC:range} are equal to $\sqrt{\BH}$.
We start with the induction step for the equivalence of the first identity in~\eqref{KRC:range}:
Assume~\eqref{KRC:range:a} holds for some $m\in\N_0$.
\makeblue{Then,} the following representations \makeblue{hold:} $(B^{m+1} -\BA^{m+1})\sqrt{\BH} =\sum_{j=0}^m \BA^j\sqrt{\BH} X_j$, with appropriate matrices $X_j\in\Cnn$, and e.g., $X_m =\BH$.
Therefore 
\[
 \range\big( (B^{m+1} -\BA^{m+1})\sqrt{\BH}\big)
 \subset \range\{\sqrt{\BH},\ \BA\sqrt{\BH},\ \ldots,\ \BA^m\sqrt{\BH}\}
\]
which implies
 \begin{align*}
 &\range\{\sqrt{\BH},\ B\sqrt{\BH},\ \ldots,\ B^m\sqrt{\BH},\ B^{m+1}\sqrt{\BH}\}
\\
 &=\range\{\sqrt{\BH},\ \BA\sqrt{\BH},\ \ldots,\ \BA^m\sqrt{\BH},\ B^{m+1}\sqrt{\BH}\} 
\\
 &=\range\{\sqrt{\BH},\ \BA\sqrt{\BH},\ \ldots,\ \BA^m\sqrt{\BH},\ \BA^{m+1}\sqrt{\BH}\} \,.
 \end{align*}
Hence, $\range\{\sqrt{\BH},\ \BA\sqrt{\BH},\ \ldots,\ \BA^{m+1}\sqrt{\BH}\} =\range\{\sqrt{\BH},\ B\sqrt{\BH},\ \ldots,\ B^{m+1}\sqrt{\BH}\}$ follows. 

In the same way, for $m\in\N_0$, the identity 
\[
 \range\{\sqrt{\BH},\ \BA\sqrt{\BH},\ \ldots,\ \BA^m\sqrt{\BH}\} 
=\range\{\sqrt{\BH},\ B^*\sqrt{\BH},\ \ldots,\ (B^*)^m\sqrt{\BH}\}
\]
is proven.
 
Since the equivalence of $\range\{\sqrt{\BH},\ \BA\sqrt{\BH},\ \ldots,\ \BA^m\sqrt{\BH}\}$ and $\range\{C_0,C_1,\ \ldots,\ C_m\}$ follows very similarly, we only give the key steps:
Note that the term of $C_{j+1}$ with the leading maximum exponent of $\BA$ is of the form $(-\BA)^{j+1}\sqrt\BH$.
This allows to compute $C_{m+1} -(-\BA)^{m+1}\sqrt{\BH} =\sum_{j=0}^m \BA^j\sqrt{\BH} X_j$, with appropriate matrices $X_j\in\Cnn$ and, e.g., $X_m =(-1)^m (m+1)\BA$.
Therefore 
\[
 \range\big( C_{m+1} -(-\BA)^{m+1}\sqrt{\BH}\big) \subset \range\{\sqrt{\BH},\ \BA\sqrt{\BH},\ \ldots,\ \BA^m\sqrt{\BH}\}
\]
and the rest of the proof is as before.
\end{proof}

Next we shall present two (equivalent) variants of Definition~\ref{def-HCIndex-full} for the hypocoercivity index that are related to similar concepts in the literature: 
In~\cite[\S 2]{ArEr14}, the definition of the hypocoercivity index of a degenerate Fokker--Planck equation with linear drift involves a positive stable matrix $B$ and its Hermitian part $\BH \geq0$ (rather than the Hermitian part and anti-Hermitian part of $B$ as in Definition~\ref{def-HCIndex-full} here).
In order to connect these two situations we shall establish the equivalence of these two definitions in the subsequent lemma.

The condition $T_m>0$ from Definition~\ref{def-HCIndex-full} can also be related to \emph{H\"ormander's ``rank $r$'' bracket condition} for hypoellipticity, cf.\ \cite{Ho67}. 
In particular, in \cite{Vi09} iterated commutators were used to establish hypocoercivity of kinetic PDEs by constructing an appropriate Lyapunov functional. 
In Lemma~\ref{HCI:AE}, Equation~\eqref{HCC:c}, below we shall mimic condition (3.5) of \cite{Vi09} for the ODE-system~\eqref{ODE:B}.
\begin{lm}\label{HCI:AE}
Let $B\in\Cnn$ be positive conservative-dissipative.
Consider the following four hypocoercivity conditions:
 \begin{subequations}\label{HCC} 
 \begin{align}
  \exists m\in\N_0: &\quad 
  T_m :=\sum_{j=0}^m \BA^j \BH (\BA^*)^j >0\ , 
  \label{HCC:a}
\\
  \exists m\in\N_0: &\quad
  \widetilde{T}_m :=\sum_{j=0}^m B^j \BH (B^*)^j >0 \ ,
  \label{HCC:b} 
\\
  \exists m\in\N_0: &\quad
  \ttT_m :=\sum_{j=0}^m (B^*)^j \BH B^j >0 \ ,
  \label{HCC:b'} 
\\ 
  \exists m\in\N_0: &\quad
  \widehat{T}_m :=\sum_{j=0}^m C^*_j C_j >0 
  \text{ with } C_0:=\sqrt{\BH};\ C_{j+1}:=[C_j,\BA]\makeblue{=C_j B_A -B_A C_j},\ j\in\N_0\ .\label{HCC:c} 
 \end{align}
\end{subequations}
The conditions~\eqref{HCC:a}--\eqref{HCC:c} are equivalent in the sense that, if there exists $m\in\N_0$ such that one condition holds, then the other three conditions hold as well for the same $m$. 
\end{lm}

\begin{proof}
According to~\cite[Lemma 2.3]{AAS15}, each of these four matrix inequalities~\eqref{HCC:a}--\eqref{HCC:c} is equivalent to the corresponding Kalman rank conditions~\eqref{KRC:a}--\eqref{KRC:c} where we used in the last case that $C_j^*=C_j$ (verifiable by a simple induction).
Moreover, the Kalman rank conditions~\eqref{KRC:a}--\eqref{KRC:c} are equivalent due to Lemma~\ref{lem:Kalman}.
\end{proof}
 
This lemma shows that the hypocoercivity index can equally be defined as the smallest integer $m\in\N_0$ such that $\widetilde{T}_m >0$, $\ttT_m >0$, or $\widehat{T}_m >0$. 
Hence, it also gives the smallest necessary number of iterated commutators of $\BH\ge0$ with the matrix $\BA$ such that their ranges span all of $\C^n$ --- in the spirit of H\"ormander's hypoellipticity theorem. 

As an example we consider two matrices with the same Hermitian part~$\BH=\diag(0,0,1,1)$ such that $\rank(\ker(\BH))=2$ but two different anti-Hermitian parts such that one and, respectively, two iterated commutators are needed:
\begin{ex}
Consider
\begin{align*}
B^{(1)}
&=
\begin{pmatrix}
 0 & 0 & 0 & 1 \\
 0 & 0 & 1 & 0 \\
 0 &-1 & 1 & 0 \\
-1 & 0 & 0 & 1 
\end{pmatrix}
\text{ such that } 
&\BA^{(1)} 
&=
\begin{pmatrix}
 0 & 0 & 0 & 1 \\
 0 & 0 & 1 & 0 \\
 0 &-1 & 0 & 0 \\
-1 & 0 & 0 & 0 
\end{pmatrix}
 , \quad 
&\big[\sqrt {\BH^{(1)}},\BA^{(1)} \big] 
&=
\begin{pmatrix}
 0 & 0 & 0 &-1 \\
 0 & 0 &-1 & 0 \\
 0 &-1 & 0 & 0 \\
-1 & 0 & 0 & 0 
\end{pmatrix}
 ,
\intertext{and} 
B^{(2)} 
&=
\begin{pmatrix}
 0 & 1 & 0 & 0 \\
-1 & 0 & 1 & 0 \\
 0 &-1 & 1 & 0 \\
 0 & 0 & 0 & 1 
\end{pmatrix}
\text{ such that }
&\BA^{(2)}
&=
\begin{pmatrix}
 0 & 1 & 0 & 0 \\
-1 & 0 & 1 & 0 \\
 0 &-1 & 0 & 0 \\
 0 & 0 & 0 & 0 
\end{pmatrix}
 ,
&\big[\sqrt {\BH^{(2)}},\BA^{(2)}\big] 
&=
\begin{pmatrix}
0 & 0 & 0 & 0 \\
0 & 0 & -1 & 0 \\
0 & -1 & 0 & 0 \\
0 & 0 & 0 & 0 
\end{pmatrix}
 .
\end{align*}
Hence $\rank\big\{\sqrt {\BH^{(1)}},\,\big[\sqrt {\BH^{(1)}},\BA^{(1)} \big] \big\}=4$, but $\rank\big\{\sqrt {\BH^{(2)}},\,\big[\sqrt {\BH^{(2)}},\BA^{(2)} \big] \big\}=3$.
Thus, by Lemma~\ref{lem:Kalman}, $\mHC(B^{(1)})=1$ and $\mHC(B^{(2)})=2$. 
\end{ex}

While finite dimensional positive conservative-dissipative matrices are hypocoercive if and only if they have a finite hypocoercivity index, this is not true in the infinite dimensional case: 
\begin{ex}
Consider a block-diagonal ``ODE'', with each block of the form
\begin{equation}\label{E_k}
 E_k
=\begin{pmatrix}
 0  & 1 & & & \\
-1 & \ddots & \ddots & & \\
   & \ddots & \ddots & 1 & \\
   & & -1 & 0 & 1 \\
   & & & -1 & 1
 \end{pmatrix} \in \R^{k\times k}
\ , \quad
(E_k)_H =\diag(0,0,\ldots,0,1)
\ , \quad
(E_k)_A =\toeplitz(-1,0,1)
\ .
\end{equation}
Then, the matrices $E_k$, $k\in\N$ are positive conservative-dissipative.   Evidently, $(E_k)_H \bv  = 0$ if and only if $\bv  = (v_1,\dots,v_{k-1},0)$.
For such a vector, 
\[
 (E_k)_A \bv  = (v_2,v_3-v_1,\dots, -v_{k-2},-v_{k-1})\ .
\]
A simple inductive argument shows that if $(E_k)_A \bv  = \lambda \bv $ for any $\lambda \neq 0$, then $\bv  = 0$. 
By Proposition~\ref{prop:border}, $E_k$ has no eigenvalue on $i\R$.
Hence, the matrices~$E_k$, $k\in\N$ are hypocoercive, satisfying the estimate 
\begin{equation} \label{est:Ek}
 \| e^{-E_k t} \|_2 \leq c_k e^{-\lambda_k t}
\ , \qquad
 t\geq 0
\ .
\end{equation}
for constants $c_k\geq 1$ and $\mu_k>0$, $k\in\N$.
We now make a simple rescaling, defining $\tE_k := r_k E_k$ for $r_k>0$ to be chosen below, and consider the propagator $P(t)$ for $\diag(\tE_1,\tE_2,\ldots)$, and let $\|\cdot\|_2$ be the spectral norm on $\ell^2$.
By \eqref{est:Ek},
\[
 \|e^{-\tE_k}\|_2 
= \|e^{-r_kE_k}\|_2 
\leq c_k e^{-r_k\mu_k} 
\leq \frac1e \quad \text{for}\quad r_k 
= \frac{1+ \log c_k}{\mu_k}\ .
\]
Therefore, making this choice of $r_k$, $\|P(1)\|_2 \leq 1/e$, so that \eqref{propagator_norm:waiting_time} is satisfied for $t_0 =1$. 
Thus this infinite dimensional system is hypocoercive. 

We now show that the combined system has an infinite hypocoercivity index. 
For this we can ignore the scaling and work with the original matrices $E_k$.
Let $\be_j$ denote the $j$-th standard basis vector in $\C^k$. 
Then, $(E_k)_H =\be_k \be_k^*$ and $T_m$ as specified in~\eqref{HCI:Tm} is given by
\[
 T_m = \sum_{j=0}^m (-1)^j (E_k)_A^j \be_k \be_k^* (E_k)_A^j\ .
\]
For $m=1,\ldots,k-1$, it is evident that
\[
 |\be_m^* (E_k)_A^j \be_k|^2 = 0 
 \quad \text{for}\quad j < k- m 
 \quad \text{but}\quad |\be_m^* (E_k)_A^{k-m} \be_k|^2 = 1 \ .
\]
Hence, $\mHC(E_k) =k-1$.
Thus, the combined system $\diag(E_1,E_2,\ldots)$ can not have a finite $\mHC$.  
\end{ex}

\subsection{Short time decay of conservative-dissipative ODE systems}
\label{sec:HC-interpr}
Here we shall prove that the hypocoercivity index of a conservative-dissipative ODE system characterizes the decay of its propagator norm for short time. 
We denote the solution semigroup pertaining to~\eqref{ODE:B} by $P(t):=e^{-B t}\in \Cnn$, and its spectral norm by $\|P(t)\|_2 := \sup\{\|P(t){\bf x}\|_2 \ :\ \|{\bf x}\|_2 =1\}$, which is also the largest singular value of $P(t)$.
Its short time decay is related to the hypocoercivity index as follows:
\begin{thm}\label{th:HC-decay}
Let the ODE system~\eqref{ODE:B} be conservative-dissipative with (positive conservative-dissipative) matrix $B\in\Cnn$. 
\begin{enumerate}[label=(\alph*)]
\item \label{th:HC-decay:a}
The (positive conservative-dissipative) matrix~$\mB$ is hypocoercive (with hypocoercivity index $\mHC\in\N_0$)
if and only if
\begin{equation}\label{short-t-decay}
  \|e^{-\mB t}\|_2 = 1-ct^a+\bigO(t^{a+1})\quad\text{ for } t\in[0,\epsilon), 
\end{equation}
for some $a,c,\epsilon>0$. In this case, necessarily $a= 2m_{HC}+1$.
%
\item \label{th:HC-decay:b} 
\makered{
Moreover, for $\mHC\geq 1$, the optimal multiplicative factor~$c$ in~\eqref{short-t-decay} is given by
\begin{equation}\label{c:mHC}
\begin{split}
 c
&= \frac1{({2{\mHC}+1})! \binom{2{\mHC}}{\mHC}} \min_{\bx\in\ker\big(\tttT_{\mHC-1}\big), \ \|\bx\|=1} \ip{\bx}{(B^*)^\mHC \BH B^\mHC \bx} 
\\
&= \frac1{({2{\mHC}+1})! \binom{2{\mHC}}{\mHC}} \min_{\bx\in\ker\big(\tttT_{\mHC-1}\big), \ \|\bx\|=1} \ip{\bx}{(\BA^*)^\mHC \BH \BA^\mHC \bx} ,
\end{split}
\end{equation}
and for $\mHC=0$ we have $c=\min_{\|\bx\|=1} \ip{\bx}{\BH\bx}$.
}
\end{enumerate}
\end{thm}
\begin{proof}
\makered{Since the proof for the coercive case, i.e. $\mHC=0$, is trivial, we shall confine ourselves now to $\mHC\geq 1$:
}

\smallskip
Part~\ref{th:HC-decay:a}: 
For sufficiently small time $t_0>0$, there exists a real analytic function $\Phi: [0,t_0]\to\R$ such that 
\begin{equation}\label{prop-norm-analytic}
  \|P(t)\|_2 =\Phi(t)\quad \mbox{ for all } t\in[0,t_0]\,,
\end{equation}
e.g. see~\cite[Lemma 1]{Ko01}.
\makeblue{Alternatively, the statement can be derived from \cite[Part II.\S6]{Ka95} or} \cite[Theorem 4.3.17]{HiPr05}.\\

\makeblue{{\bf For the forward direction} we assume that $B$ has a finite HC-index $m_{HC}\in\N$. Hence} we are left to prove that the Taylor expansion of $\|P(t)\|_2$ has the form~\eqref{short-t-decay}. 
\makeblue{This proof will be split into two separate parts, the lower and the (technically more subtle) upper bound. }

\medskip
\makeblue{\underline{Lower bound:}} First, we shall prove that there exists~$c_1>0$ such that
\begin{subequations} \label{short-t-decay:lower:1+2}
\begin{equation} \label{short-t-decay:lower}
 \|P(t)\|_2 \geq 1-c_1 t^a +\bigO(t^{a+1})
 \quad\text{as } t\to 0^+ \ 
\end{equation}
\makeblue{or, equivalently,}
\begin{equation} \label{short-t-decay:lower2}
\makeblue{
 \|P(t)\|_2^2 \geq 1-2c_1 t^a +\bigO(t^{a+1})
 \quad\text{as } t\to 0^+ \ 
}
\end{equation}
\end{subequations}
\makeblue{holds with $a=2\mHC+1$.} \footnote{\makeviolet{In inequalities, such as~\eqref{short-t-decay:lower:1+2}, the Landau symbol $\bigO$ is used in the sense that there exist constants $\epsilon,M>0$ such that $\|P(t)\|_2 \geq 1-c_1 t^a -M\ t^{a+1}$ for $t\in[0,\epsilon)$.}}
\makeblue{To this end}
suppose $\bx_0$ is any unit vector such that for some $m\in \N$ \makeblue{(\emph{not} necessarily $m=m_{HC}$), } 
\begin{equation}\label{bo3}
\bx_0 
 \in \ker \left(\sum_{j= 0}^{m-1} (B^*)^j B_H B^j\right) 
 =\ker\big(\ttT_{m-1}\big) \,,
\end{equation}
which is equivalent to
\begin{equation}\label{bo1}
\|\sqrt{B_H}B^j \bx_0\|_2 =0 \quad{\rm for} \quad 0 \leq j \leq m-1 \  .
\end{equation}
\makeblue{If such an $\bx_0$ exists, we will show for the corresponding trajectory} that
\begin{equation}\label{bo2}
\|P(t)\bx_0\|_2^2 
= 1 - t^{2m+1} \frac{2}{(2m+1)!} \binom{2m}{m} \ip{ \bx_0}{ (B^*)^m B_H B^m \bx_0} + \mathcal{O}(t^{2m+2})\ .
\end{equation}
Note that our hypotheses allow the possibility that $\langle\bx_0, (B^*)^m B_H B^m \bx_0\rangle= 0$ in which case we have simply $\|P(t)\bx_0\|_2^2 = 1 +\mathcal{O}(t^{2m+2})$.
 
\makeblue{Now take $\bx_0$ as in Equation~\eqref{bo3} with $\|\bx_0\|=1$, and define $q(t) := \|P(t) \bx_0\|^2 $. Standard theory gives that (see also Equation~\eqref{eq:energy})
\begin{equation} \label{eq:1}
 {q}'(t) 
= -2 \|\sqrt{B_H} P(t)\bx_0\|^2
= -2 \ip{ \sqrt{B_H} P(t)\bx_0}{ \sqrt{B_H} P(t)\bx_0}.
\end{equation}
Using the assumption on $\bx_0$, we have that $\sqrt{B_H} P(t) \bx_0 = 
\sum_{k=m}^{\infty} \frac{t^k}{k!} \sqrt{B_H} (-B)^k \bx_0$. 
Substituting this in \eqref{eq:1} gives
\begin{equation} \label{eq:2}
 {q}'(t) 
= -2 \frac{t^{2m}}{m!\cdot m!} \ip{ \sqrt{B_H} B^m \bx_0}{ \sqrt{B_H} 
B^m \bx_0} + \sum_{k=2m+1}^{\infty} u_k t^k 
\end{equation}
for some scalars $u_k$. 
Since this sequence converges absolutely (The expression in \eqref{eq:1} is real analytic), we can integrate this expression term-wise to obtain~$q(t)$. 
Using that $q(0)=\|\bx_0\|^2=1$, we find
\begin{equation}\label{eq:3}
 q(t) 
= 1 -\frac{2}{2m+1} \frac{t^{2m+1}}{m!\cdot m!} \ip{ \sqrt{B_H} B^m \bx_0}{ \sqrt{B_H} B^m \bx_0} + \sum_{k=2m+1}^{\infty} \frac{u_k}{k+1} t^{k+1},
\end{equation}
which shows \eqref{bo2}. 
}

\makeblue{Now, since the positive conservative-dissipative matrix $B$ is hypocoercive with HC-index $\mHC \in\N_0$, there exists a normalized vector $\bx_0$ such that
\begin{equation}\label{x0}
\|\sqrt{B_H}B^j \bx_0\|_2 =0 \quad\text{for } 0 \leq j \leq \mHC-1,
\quad\text{and } 
\sqrt{\BH} B^{\mHC}\bx_0\ne 0\,, 
\end{equation}
but none that would satisfy instead also $\sqrt{\BH} B^{\mHC}\bx_0 =0$. 
Hence, for such $\bx_0 \in \ker\big(\ttT_{\mHC-1}\big)$ the identity 
\eqref{bo2} holds with $m=\mHC$. 
}
%
%
%
%
\makeblue{Finally, taking the supremum over all initial conditions $\bx_0\in\ker\big(\ttT_{\mHC-1}\big),\ \|\bx_0\|_2=1$ (which is a closed set) yields the estimates~\eqref{short-t-decay:lower} and \eqref{short-t-decay:lower2} with $a=2\mHC+1$ and }
\begin{equation} 
\begin{split} \label{c:lower_bound}
 c_1 
:=& \frac12 \min_{ \bx_0\in\ker\big(\tttT_{\mHC-1}\big),\ \|\bx_0\|_2=1} \tfrac2{(2\mHC+1)!}\binom{2\mHC}{\mHC} \ip{ \bx_0}{ (B^*)^{\mHC}B_H B^\mHC \bx_0}
\\
=& \frac1{(2\mHC+1)!}\binom{2\mHC}{\mHC}  \min_{ \bx_0\in\ker\big(\tttT_{\mHC-1}\big),\ \|\bx_0\|_2=1} \ip{ \bx_0}{ (\BA^*)^{\mHC}B_H \BA^\mHC \bx_0}
\ .
\end{split}
\end{equation}
In the last step we used
\begin{equation} \label{Bmx:BAmx}
 B^{\mHC}\bx_0 
=\BA B^{\mHC-1}\bx_0  
=\ldots
=\BA^{\mHC}\bx_0 
\,,
\end{equation}
since all other terms vanish due to~$\bx_0\in\ker\big(\ttT_{\mHC-1}\big)$.
Due to~\eqref{x0}, $c_1>0$. 

\makered{In the proof of Part~\ref{th:HC-decay:b} we shall actually improve this constant by the factor $\binom{2\mHC}{\mHC}^{-2}$, see Lemma~\ref{lm:c:upper_bound}.
}

\medskip
\makeblue{\underline{Upper bound:}} Second, we prove that there exists~$c_2>0$ such that
\begin{subequations} \label{short-t-decay:upper:1+2}
\begin{equation} \label{short-t-decay:upper}
 \|P(t)\|_2 \leq 1-c_2 t^a +\bigO(t^{a+1})
 \quad\text{as } t\to 0^+ \ 
\end{equation}
\makeblue{or, equivalently,}
\begin{equation} \label{short-t-decay:upper2}
\makeblue{
 \|P(t)\|_2^2 \leq 1-2c_2 t^a +\bigO(t^{a+1})
 \quad\text{as } t\to 0^+ \ 
}
\end{equation}
\end{subequations}
\makeblue{holds with $a=2\mHC+1$.}
\footnote{In inequalities, such as~\eqref{short-t-decay:upper:1+2}, the Landau symbol $\bigO$ is used in the sense that there exist constants $\epsilon,M>0$ such that $\|P(t)\|_2 \leq 1-c_2 t^a +M\ t^{a+1}$ for $t\in[0,\epsilon)$.}
Here, we consider the case $\mHC =1$ whereas the general case $\mHC\in\N$ is treated in Appendix~\ref{app:short-time-decay}.

\makeblue{If matrix~$B$ has hypocoercivity index $\mHC=1$,  there exists $\kappa>0$ such that 
\begin{equation}\label{HC:kappa:mHC_1}
 \ttT_\mHC 
=\sum_{j=0}^{\mHC} (B^*)^j \BH B^j 
=\BH + B^* \BH B 
\geq \kappa I >0 \,. 
\end{equation}
Since $\mHC =1$, $B^*B_HB$ is positive definite on $\ker(B_H)$. 
For $\bx\in\Cn$ with $\|\bx\|=1$, we define
\begin{equation} \label{lambda_x:mu_x:mHC_1}
 \lambda_\bx :=\ip{\bx}{\BH \bx}\geq 0\ , \quad
 \mu_\bx :=\ip{\bx}{B^* \BH B \bx}\geq 0\ , \quad
 \text{such that } \lambda_\bx +\mu_\bx \geq \kappa >0\ .
\end{equation}

\medskip
Note that $\|P(t)\bx_0\|_2^2 = \ip{ \bx_0}{ Q(t)\bx_0}$ where
\begin{equation}\label{Q-expans}
 Q(t)
:= P^*(t)\,P(t) = e^{-B^* t}\, e^{-B t}
= \sum_{j=0}^{\infty} \frac{t^j}{j!} \sum_{k=0}^j \binom{j}{k} (-B^*)^k (-B)^{j-k} \ .
\end{equation}
Since ${ \big\Vert \sum_{k=0}^j \binom{j}{k} (-B^*)^k (-B)^{j-k}\big\Vert_2} \leq (2\|B\|_2)^j$, the Taylor series for the matrix family $Q(t)$ converges uniformly on bounded $t$-intervals.
Hence we have 
\begin{equation}\label{Q:norm}
 \|P(t)\|_2^2 
=\|Q(t)\|_2 
=\|Q_j(t)\|_2 +\bigO(t^{j+1})\,,
\end{equation}
where $Q_j(t)$ denotes the Taylor expansion~\eqref{Q-expans}, but truncated after the $t^j$-term.
Consider
\begin{equation}\label{Q3}
Q_3(t)
=\sum_{j=0}^{3} \frac{t^j}{j!}\ U_j 
=I +t U_1 +\frac{t^2}{2!} U_2 +\frac{t^3}{3!} U_3 .
\end{equation}
Then, its spectral norm satisfies
\begin{equation}\label{Q3:norm}
 \|Q_3(t)\|_2
:= \sup_{\|\bx_0\|=1} \|Q_3(t)\bx_0\|_2 
= \sup_{\|\bx_0\|=1} \ip{ \bx_0}{ Q_3(t)\bx_0} ,
\end{equation}
for sufficiently small $t\geq 0$.
The latter identity holds, since the self-adjoint matrix $Q_3(t)$ satisfies $Q_3(0)=I$.

}

Let $\bx$ be a unit vector, \makeblue{we estimate each expression $\ip{\bx}{U_j \bx}$, $j=1,2,3$ separately as
\[
 \ip{ \bx}{ U_1\bx} 
=-2\ip{ \bx}{ B_H\bx} 
=-2 \lambda_\bx \quad\text{using}\quad 
 \lambda_\bx 
= \ip{ \bx}{ B_H \bx} \geq 0 \ .
\]
Next, since $U_2 = 2(B^*B_H + B_H B)$, so that by the Cauchy--Schwarz inequality,
\begin{equation*}
 \ip{ \bx}{ U_2\bx}
= 2\ip{ \bx}{ (B^*B_H + B_H B)\bx} 
\leq 4\| \sqrt{B_H}\bx\|  \| \sqrt{B_H} B\bx\| 
= 4 \sqrt{\lambda_\bx} \sqrt{\mu_\bx}
\end{equation*}
using~\eqref{lambda_x:mu_x:mHC_1}. In the same way, we derive for $U_3 = -2(B_H B^2 +2B^* B_H B +(B^*)^2 B_H)$ that
\begin{multline*}
 \ip{ \bx}{ U_3\bx}
\leq -4\ip{ \bx}{ B^*B_H B \bx} -2\ip{ \bx}{ (B_H B^2 +(B^*)^2 B_H)\bx}
 \\
\le  -4\mu_\bx +4\norm{\sqrt{B_H} \bx}\norm{\sqrt{B_H} B^2 \bx}
= -4\mu_\bx +4\sqrt{\lambda_\bx}\norm{\sqrt{B_H} B^2 \bx} .
\end{multline*}
Altogether then,
\begin{equation} \label{est:Q3}
\begin{split}
 \langle \bx, Q_3(t)\bx\rangle
&\leq 1 -2 \lambda_\bx\ t +2\sqrt{\lambda_\bx} \sqrt{\mu_\bx}\ t^2 -\frac23 \mu_\bx\ t^3 +\frac23 \sqrt{\lambda_\bx}\norm{\sqrt{B_H} B^2}\ t^3 
 \\
&= 1 -2t \Big( \sqrt{\lambda_\bx} -\frac{\sqrt{\mu_\bx}}2 t\Big)^2 -\frac16 \mu_\bx\ t^3 +\frac23 \sqrt{\lambda_\bx}\norm{\sqrt{B_H} B^2}\ t^3 .
\end{split}
\end{equation}
To estimate~\eqref{est:Q3},
we will distinguish two cases for $\sqrt{\lambda_\bx}\in[0,\|\sqrt{B_H}\|]$:
If $\sqrt{\lambda_\bx}$ is sufficiently small we will use the third term in~\eqref{est:Q3} to compensate the non-negative fourth term in \eqref{est:Q3}, while in the other case we will use the second term to do so.

\underline{Case a:}
If 
\begin{equation}\label{alpha}
\sqrt{\lambda_\bx}\leq \alpha
 \quad\text{with }
 \alpha :=\min\Bigg(1,\ \frac{\kappa}{2\big(1+4\norm{\sqrt{B_H}B^2}\big)}\Bigg) >0 ,
\end{equation}
then the following estimates hold:
Due to~\eqref{lambda_x:mu_x:mHC_1} and $\sqrt{\lambda_\bx}\leq 1$ we deduce that $\mu_\bx \geq \kappa -\lambda_\bx \geq \kappa -\sqrt{\lambda_\bx}$ which implies that
\begin{equation} \label{est:Q3:Case_a}
 \langle \bx, Q_3(t)\bx\rangle
\leq 1 -\frac16 \mu_\bx\ t^3 +\frac23 \sqrt{\lambda_\bx}\norm{\sqrt{B_H} B^2}\ t^3 
\leq 1 -\frac{\kappa}{12} \ t^3 ,
 \qquad \text{for } t\geq 0\ .
\end{equation}

\underline{Case b:}
If $\alpha\leq \sqrt{\lambda_\bx}\leq \norm{\sqrt{B_H}}$ then we restrict the time interval to obtain a similar estimate.
Using $\alpha$ from~\eqref{alpha}, define 
\begin{equation} \label{t_*}
 t_1 := \frac{\alpha}{\sqrt{\norm{B^* B_H B}}} .
\end{equation}
For $t\in[0,t_1]$, we use $0\leq\mu_\bx\leq\norm{B^* B_H B}$ to derive 
\[
 \sqrt{\lambda_\bx} -\frac{\sqrt{\mu_\bx}}2 t
\geq \alpha -\frac{\sqrt{\mu_\bx}}2 t_1
\ge \frac{\alpha}2
> 0 .
\]
Consequently, the expression in~\eqref{est:Q3} can be estimated as
\begin{equation*} 
\begin{split}
\langle \bx, Q_3(t)\bx\rangle
&\leq 1 -\frac{\alpha^2}2 t +\frac23 \norm{\sqrt{B_H}}\ \norm{\sqrt{B_H} B^2}\ t^3 
 \\
&\leq 1 -t^3 \Big( \frac{\alpha^2}2 \frac1{t_1^2} -\frac23 \norm{\sqrt{B_H}}\ \norm{\sqrt{B_H} B^2}\Big) .
\end{split}
\end{equation*}
Then, there exists $t_*\in(0,t_1]$ such that 
\begin{equation} \label{est:Q3:Case_b}
\langle \bx, Q_3(t)\bx\rangle
\leq 1 -\frac{\kappa}{12} \ t^3 ,
 \qquad \text{for } t\in[0,t_*]\ .
\end{equation}
Together the estimates~\eqref{est:Q3:Case_a} and~\eqref{est:Q3:Case_b} as well as  \eqref{Q:norm}, \eqref{Q3:norm}
}
prove the upper bound~\eqref{short-t-decay:upper2} with $a=3$ for $\mHC =1$, e.g. with \makeblue{$c_2 :=\kappa/24$}. 
For $\mHC \geq 2$, see \makeblue{Lemma~\ref{lm:Q:norm} in Appendix~\ref{app:short-time-decay}}. 

\medskip
To sum up, due to~\eqref{short-t-decay:lower} and Lemma~\ref{lm:Q:norm}, there exist \makered{constants} $c_1, c_2>0$ such that $a=2\mHC+1$ and
\begin{equation} \label{short-t-decay:sandwich}
 1-c_1 t^a +\bigO(t^{a+1})
\leq \| P(t)\|_2
\leq 1-c_2 t^a +\bigO(t^{a+1})
\quad\text{as } t\to 0^+ 
\ ,
\end{equation}
e.g. choosing $c_1$ as in~\eqref{c:lower_bound} and $c_2$ as in Lemma~\ref{lm:Q:norm}, respectively.
Moreover, due to \eqref{prop-norm-analytic}, $\|P(t)\|_2$ is analytic on $[0,t_0]$.
This implies that the propagator norm satisfies~\eqref{short-t-decay} for some $c\in\R$, which satisfies $0 <c_2 \leq c \leq c_1$.

\medskip
\makeblue{{\bf For the reverse direction}} suppose that~\eqref{short-t-decay} is satisfied for some constant~$c>0$ and an exponent $a>0$. 
Then, \makeblue{as we shall show,} for some least finite value of $m\in\N_0$, $T_m =\sum_{j=0}^{m} \BA^j \BH(\BA^*)^j >0$ has to hold, or equivalently,
\[
 \ttT_m 
=\sum_{j=0}^{m} (B^*)^j \BH B^j >0 \,.
\]
Otherwise, for arbitrarily large values of $\tm\in \N$, we could find unit vectors $\bx_0$ such that $\ttT_{\tm-1}\bx_0 = 0$, and then by the first part of this proof, we would have for such a vector
$\|P(t)\bx_0\|_2^2 \geq 1 -c_{\tm}t^{2\tm+1}+ \mathcal{O}(t^{2\tm+2})$ with some $c_{\tm}\ge0$, because of \eqref{eq:3}. 
But for sufficiently large $\tm$, this is incompatible with~\eqref{short-t-decay}.
Thus we conclude that, whenever \eqref{short-t-decay} is valid for {\em any} $a>0$, $B$ has a finite hypocoercivity index~$\mHC\in \N_0$, and then necessarily, $a\in2\N_0+1$.

\medskip
Part~\ref{th:HC-decay:b}: 
\makered{So far we have} proved that the propagator norm satisfies~\eqref{short-t-decay:sandwich}, e.g. choosing $c_1$ as in~\eqref{c:lower_bound} and \modif{$c_2$ as the lower bound of~$c$ in~\eqref{c:lower_bound:mHC}}, respectively.
Using the improved upper bound~$c_1$ in~\eqref{c:upper_bound:mHC} and lower bound~$c_2$ in~\eqref{c:lower_bound:mHC} for the multiplicative constant $c$ we realize that $c_2 =c =c_1$ such that~\eqref{c:mHC} holds.
In the final identity of~\eqref{c:mHC} we used again~\eqref{Bmx:BAmx} to reveal that $c$ is proportional to $\BA^{2m_{HC}}$.
This finishes the proof.
\end{proof}

%
\makered{
For $\epsilon$-dependent ODE systems of the form~\eqref{ODE:epsA+C}, Theorem \ref{th:HC-decay}\ref{th:HC-decay:b} implies the following result: 
\begin{cl} \label{cl:HC-decay}
Consider the $\epsilon$-dependent ODE \eqref{ODE:epsA+C} with system matrix $B=\epsilon A +C$ where $\epsilon\in\R$.
If $B=\epsilon A+C$ is hypocoercive for $\epsilon\ne 0$, then the coefficient $c=c_\epsilon$ in the Taylor expansion of the propagator norm~\eqref{short-t-decay} satisfies 
\begin{equation} \label{cl:HC-decay:epsilon}
 c
=c_\epsilon 
=\epsilon^{2m_{HC}}\ \frac1{({2{\mHC}+1})! \binom{2{\mHC}}{\mHC}} \min_{\bx_0\in  \ker\big(\tttT_{\mHC-1}\big), \ \|\bx_0\|_2=1} \langle \bx_0,(A^*)^{\mHC}C A^\mHC \bx_0\rangle \ .
\end{equation}
\end{cl}
}
\begin{remark}
As already briefly mentioned in Example 2 of \cite{AAM21}, the hypocoercivity index $m_{HC}$ is upper semicontinuous w.r.t.\ the matrix $\mB$: 
An arbitrarily small perturbation of $\mB$ can lower but not increase the index. This is consistent with the result in \eqref{short-t-decay}: 
All Taylor coefficients are there of course continuous in $\mB$. 
But a small perturbation of $\mB$ may lower, but not increase the number~$a$ of the first non-vanishing monomial in this Taylor series (beyond the constant 1).
\end{remark}

\makered{
\begin{remark}
    Note that the leading exponent in \eqref{short-t-decay} can only be odd. This is related to the local behavior of trajectories that decay the worst (in the vicinity of a stationary point $t_0$ of $\|\bx(t)\|$), see the yellow curve in Fig.\  \ref{fig:envelope}, left. Since the system \eqref{ODE:B} is assumed to be conservative-dissipative, such a trajectory, of course, cannot behave locally like $1 - c(t-t_0)^a$ with $a$ even. 
\end{remark}
}

\begin{remark}
Special cases of the above theorem were pointed out to us by Laurent Miclo: 
In \S1 of \cite{MiMo13} the short time decay behavior of the Goldstein-Taylor model (a linear transport equation with relaxation term) was determined as $1-\frac{t^3}3+o(t^3)$. 
Actually, this model is a PDE. 
But since it is considered on a torus in $x$, each of its spatial Fourier modes (except of the 0-mode) satisfies a conservative-dissipative ODE system with hypocoercivity index 1 (see \cite{AAC16} for details of this modal decomposition). 
Hence, mode by mode, the result from \cite{MiMo13} is an example for Theorem~\ref{th:HC-decay}. 
For closely related BGK-models with hypocoercivity index 2 and 3 we refer to \cite{AAC18}.

In \cite{GaMi13} the short time decay behavior of a kinetic Fokker--Planck equation on the torus in $x$ was computed as $1-\frac{t^3}{12}+o(t^3)$. 
Again, in Fourier space and by using a Hermite function basis in velocity, this model can be written as an (infinite dimensional) conservative-dissipative system with hypocoercivity index 1 (see \S2.1 of \cite{GaMi13}). 
In that paper it was also mentioned that the decay exponent in~\eqref{short-t-decay} can be seen as some ``order of hypocoercivity'' of the generator.

For degenerate Fokker--Planck equations, the hypocoercivity index can also be related to the regularization rate for short times: 
In \cite[Theorem A.12]{Vi09} the regularization of initial data from a weighted $L^2$ space into a weighted $H^1$ space is derived, and in \cite[Theorem A.15]{Vi09}, \cite[Theorem 4.8]{ArEr14} it is generalized to entropy functionals and their corresponding Fisher informations. 
In all these cases the regularization rate is $t^{-a}$ with $a=2\mHC+1$ (somewhat related to Theorem~\ref{th:HC-decay} above).
\end{remark}

%
%

\medskip
\modif{By definition, the propagator norm of an ODE~\eqref{ODE:B} is given as the envelope of the norm of a family of solutions, see e.g. \cite{AAS19} and Figure~\ref{fig:envelope}. But, maybe surprisingly, even its precise short time behavior is \emph{not} given by the norm of any specific solution. This is illustrated in the following example.}

\begin{ex}\label{ex:envelope}
We consider ODE~\eqref{ODE:B} with matrix 
\begin{equation}\label{B:envelope}
 B =\begin{pmatrix} 1 & -3/10 \\ 3/10 & 0 \end{pmatrix} .
\end{equation}
The eigenvalues of $B$ are $\lambda_1 =1/10$ and $\lambda_2 =9/10$, and the eigenvalues of $\BH$ are $0$ and $1$.
Thus, matrix $B$ is hypocoercive with hypocoercivity index $\mHC=1$. 
Following (the first part of) the proof of Theorem~\ref{th:HC-decay}, solutions starting in $\bx_0$ satisfying~\eqref{x0} are used to establish the desired lower bound of the propagator norm.
The kernel of $\BH$ is one-dimensional and it is spanned by $\bx_0 =\binom{0}{1}$.
The solution of~\eqref{ODE:B} with initial condition $\bx(0)=\bx_0$ is given by
\[
 \bx(t)
 =\tfrac18 \begin{pmatrix} 3 & -3 \\ 9 & -1 \end{pmatrix} 
  \binom{e^{-t/10}}{e^{-9t/10}}
\]
and its squared norm satisfies
\[
 \|\bx(t)\|^2_2
 =\tfrac{45}{32} e^{-t/5} -\tfrac9{16} e^{-t} +\tfrac5{32} e^{-9t/5}
 \sim 1 -0.06\ t^3 +\bigO(t^4) \quad\text{for } t\to 0^+.
\]
By contrast, due to \cite[Proposition 4.2]{AAS19}, the squared propagator norm satisfies
\begin{equation} \label{ex:short-time-decay}
 \begin{split}
 \|e^{-Bt}\|^2_2 
 \leq &e^{-t} \tfrac1{16} \Big(\sqrt{(25 \cosh(8t/10)-9)^2 -16^2} +25\cosh(8t/10) -9\Big) \\
 &\sim 1 -0.015\ t^3 +\bigO(t^4) \quad\text{for } t\to 0^+.
 \end{split}
\end{equation}
Thus the propagator norm decays slower than the solution starting at the vector $\bx_0$ which satisfies~\eqref{x0} with $\mHC=1$, see also Figure~\ref{fig:envelope}.

\makered{In fact, the sharp constant $c=0.015$ in \eqref{ex:short-time-decay} cannot be obtained by any single trajectory, but rather by a family of trajectories starting at the one-parameter family of normalized initial conditions, $\bx_\tau,\,\tau\ge0$ emerging from $\bx_0$: More concretely, using 
$\bx_\tau:=\big(-\frac{3\tau}{20}, 1\big)^\top \big/\sqrt{1+\frac{9\tau^2}{400}}$ (as constructed in Lemma~\ref{lm:g} below) yields
$$
  \|e^{-Bt}\,\bx_t\|^2_2 \sim 1 -0.015\ t^3 +\bigO(t^4) \quad\text{for } t\to 0^+.
$$
Note that we used here the initial conditions $\bx_t$ (i.e.\ $\tau=t$), and that we have $\ddtau\bx_\tau(\tau=0)=\big(-\frac{3}{20},0\big)^\top=\frac12 B\bx_0$.\\ 
}
\begin{figure}
\begin{center}
\includegraphics[scale=0.42]{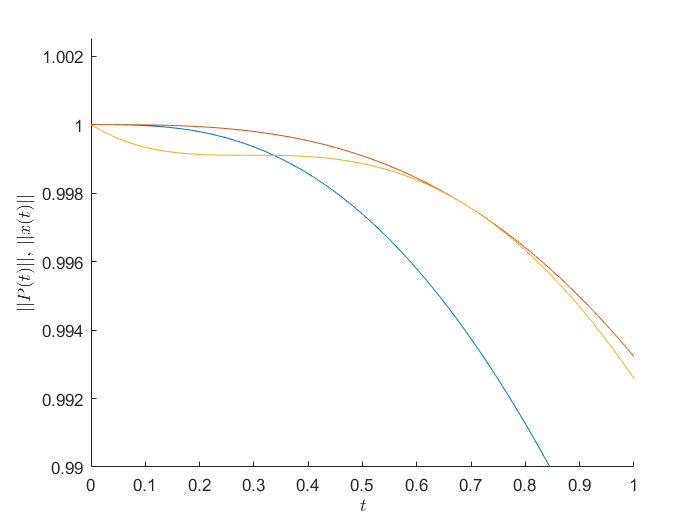}
\includegraphics[scale=0.42]{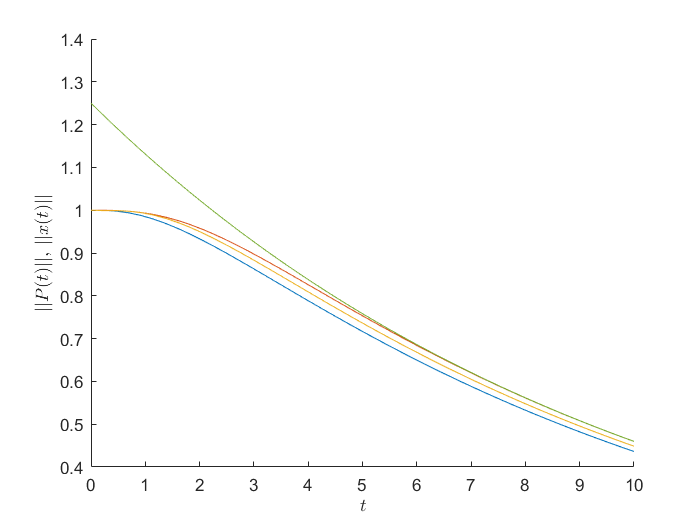}
\end{center}
\caption{Evolution of~\eqref{ODE:B} with matrix $B$ from~\eqref{B:envelope}: Comparison between the propagator norm (red line), its upper exponential envelope $1.25\exp(-t/10)$ (green line), and the norm of the solution with initial condition $\bx(0)=\binom01$ (blue line) and that with initial condition $\bx(0)=\binom{-0.1}{\sqrt{0.99}}$ (yellow line), all plotted on two time scales.}
 \label{fig:envelope}
\end{figure}
\end{ex}
%
%

%
\subsection{Numerical illustration of the short time decay and the waiting time~$t_0$}\label{subsect:num}

Next we shall illustrate the decay behavior on two examples of dimension $n=4$. 
In particular we shall consider the $\epsilon$-dependence of the three phases, the asymptotic phase close to $t=0$ (as characterized in Theorem \ref{th:HC-decay}), the intermediate phase (characterized by the waiting time $t_0$ and the exponential decay for large time (see~\eqref{HC-decay}).

\begin{ex} \label{ex:num1}
Consider the matrix family~$B_\epsilon:=\epsilon A +C$, $\epsilon\ne0$ with 
\begin{equation}
A =
\begin{pmatrix}
0 & 0 & 0 & 1 \\
0 & 0 & 1 & 0 \\
0 & -1 & 0 & 0 \\
-1 & 0 & 0 & 0 
\end{pmatrix} \,, \qquad 
C = \diag(0,0,1,1) \ ,
\end{equation}
which satisfies $m_{HC}(B_\eps)=1$.

Figure \ref{fig1} shows the spectral norm of the semigroup $P_\epsilon(t):=e^{-B_\eps t}$ as a function of time and for several values of $\epsilon$ (differing from each other by the factor $\sqrt2$). 
The numerically observed waiting times (to decrease the solution norm by the factor $1/e$) are very close to $t_0\sim 1/\epsilon^2$. 
They increase by a factor of about 2 when passing from $\epsilon$ to $\sqrt2 \epsilon$, at least in the asymptotic regime $\eps\to0$. 
A close-up of the same figure around $t=0$ shows that the asymptotic behavior of the semigroup norm is like $\|P_\epsilon(t)\|_2\sim 1-c_\epsilon\,t^3$ with $c_\epsilon= \tilde c\,\epsilon^2$, and by recalling that $m_{HC}(B_\eps)=1$ implies $a=3$ in \eqref{short-t-decay}. 
\makered{Following~\eqref{cl:HC-decay:epsilon}, the multiplicative factor is} $\tilde c=\frac{1}{12}$ (compare with e.g.~\cite{AAS19}).
\begin{figure}[htbp]
\begin{center}
\includegraphics[scale=0.4]{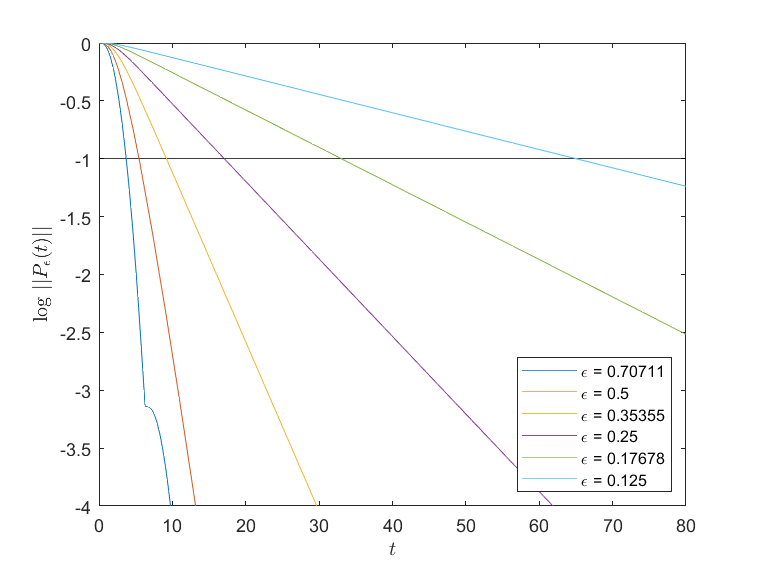}
\includegraphics[scale=0.4]{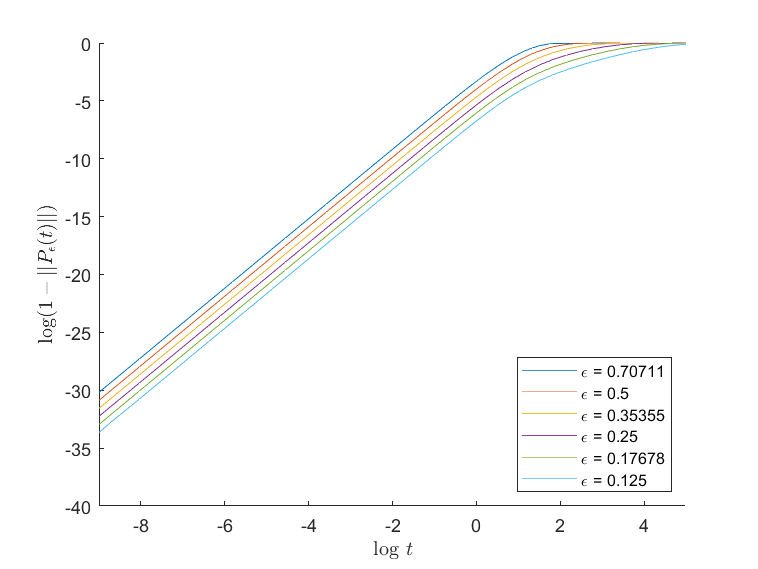}
\end{center}
\caption{The decay of $\|P_\eps(t)\|_2$ is given for six values of $\epsilon$. 
Left: 
For $t$ away from 0, this semigroup decays almost exponentially. 
With the logarithmic scale used here, the horizontal black line corresponds to $1/e$. 
The waiting times (defined as intersection with the line $1/e$) behave like $\bigO(\epsilon^{-2})$. 
We remark that the kink in the blue curve is \emph{not} a numerical artifact. \newline
Right: 
This double logarithmic plot shows $1-\|P_\eps(t)\|_2\sim c_\epsilon\,t^3$ for small time, more precisely for $t\in[e^{-9},e^5]$. 
The curves have slope 3, and $c_\epsilon= \tilde{c}\,\epsilon^2$. 
The plot also shows the quite sharp transition from the initial algebraic behavior $1-\tilde{c}\,\eps^2\,t^3$ to the exponential behavior $c_\epsilon^* \, e^{-\tilde{\mu}\,\epsilon^2\,t}$.
}
\label{fig1}
\end{figure}
\end{ex}

\begin{ex} \label{ex:num2}
This example is analogous to Example \ref{ex:num1}, but for the matrix family~$B_\epsilon:=\epsilon A +C$, $\epsilon\ne0$ with 
\begin{equation}
A =
\begin{pmatrix}
0 & 1 & 0 & 0 \\
-1 & 0 & 1 & 0 \\
0 & -1 & 0 & 1 \\
0 & 0 & -1 & 0 
\end{pmatrix} \,, \qquad 
C = \diag(0,0,0,1) \ ,
\end{equation}
which satisfies $m_{HC}(B_\eps)=3$.

Figure \ref{fig2} shows the spectral norm of the semigroup $P_\epsilon(t):=e^{-B_\eps t}$. 
The numerically observed waiting times (to decrease the solution norm by the factor $1/e$) are very close to $t_0\sim 4/\epsilon^2$.  
A close-up around $t=0$ shows that the asymptotic behavior of the semigroup norm is like $\|P_\epsilon(t)\|_2\sim 1-c_\epsilon\,t^7$ with $c_\epsilon= \tilde c\,\epsilon^6$, and by recalling that $m_{HC}(B_\eps)=3$ implies $a=7$ in \eqref{short-t-decay}. 
\makered{Following~\eqref{cl:HC-decay:epsilon}, the multiplicative factor is} $\tilde c=\frac{1}{100800}$. 

\begin{figure}[htbp]
\begin{center}
\includegraphics[scale=0.4]{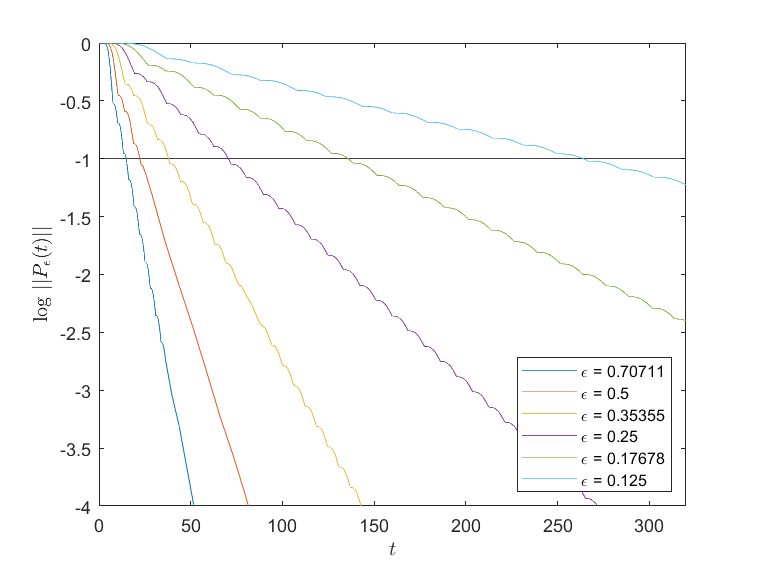}
\includegraphics[scale=0.4]{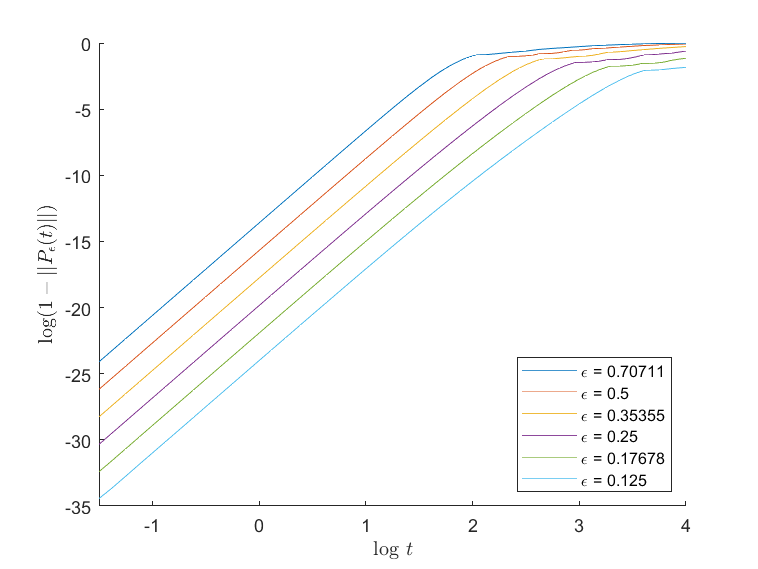}
\end{center}
\caption{The decay of $\|P_\eps(t)\|_2$ is given for six values of $\epsilon$. 
Left: 
For $t$ away from 0, this semigroup decays almost exponentially. 
With the logarithmic scale used here, the horizontal black line corresponds to $1/e$. 
The waiting times (defined as intersection with the line $1/e$) behave like $\bigO(\epsilon^{-2})$. \newline
Right: 
This double logarithmic plot shows $1-\|P_\eps(t)\|_2\sim c_\epsilon\,t^7$ for small time, more precisely for $t\in[e^{-1.5},e^4]$. 
The curves have slope 7, and $c_\epsilon= \tilde{c}\ \epsilon^6$. 
The plot also shows the quite sharp transition from the initial algebraic behavior $1-\tilde{c}\ \eps^6\,t^7$ to the exponential behavior $c_{\epsilon}^*\ e^{-\tilde{\mu} \ \epsilon^2\ t}$.
}
\label{fig2}
\end{figure}
\end{ex}

\appendix
\renewcommand{\thesection}{\Alph{section}}

\section{Auxiliary results to prove short time decay of propagator norm}
\label{app:short-time-decay}
To finish the proof of Theorem~\ref{th:HC-decay}, we shall prove the following upper bound for the propagator norm:
\begin{lm}\label{lm:Q:norm}
Let the ODE system~\eqref{ODE:B} be conservative-dissipative, and let the system matrix $B$ be (hypo)coercive with hypocoercivity index $\mHC\in\makered{\N}$.
Then, there \makeviolet{exist constants $c_2,M,t_2>0$} 
such that 
\begin{equation} \label{est:Propagator}
 \| e^{-B t} \|_2
\leq 1 -c_2 t^a +\makeviolet{M t^{a+1}} 
\qquad \text{for } t\makeviolet{\in[0,t_2]}
\ ,
\end{equation}
where $a =2 \mHC +1$.
\makered{
Moreover, the multiplicative factor~$c$ in~\eqref{short-t-decay} satisfies
\begin{equation}\label{c:lower_bound:mHC}
 c \geq c_2
\qquad\text{with }
 c_2
:= \frac1{({2{\mHC}+1})! \binom{2{\mHC}}{\mHC}} \min_{\bx\in\ker\big(\tttT_{\mHC-1}\big), \ \|\bx\|=1} \ip{\bx}{(B^*)^\mHC \BH B^\mHC \bx} .
\end{equation}
}
\end{lm}
\begin{proof}[Proof of Lemma~\ref{lm:Q:norm}]
First we note that the hypocoercivity of~$B$ implies $\BH\ne0$ due to~\eqref{HC-charact}. 
Following the proof of Theorem~\ref{th:HC-decay}, we consider
$\| e^{-B t} \|_2^2 =\lambda_{\max} \big(Q(t)\big)$ for small $t>0$, where 
\begin{equation}\label{A:Q}
Q(t)
=e^{-B^* t}e^{-Bt}
=\sum_{j=0}^{\infty} \frac{t^j}{j!}\ U_j 
\end{equation}
with
\begin{equation}\label{est:Uj} 
U_j
=(-1)^j \sum_{k=0}^j \binom{j}{k} (B^*)^k B^{j-k}\,,
\quad\text{satisfying }
\|U_j\|_2 \leq (2 \|B\|_2)^j \,, \quad j\in\N_0 \,.
\end{equation}
To compute $\lambda_{\max} \big(Q(t)\big) = \max_{\|\bx\|_2=1} \bx^* Q(t) \bx$, we consider the $t$-dependent function $g(\bx;t) :=\bx^* Q(t) \bx -1$ with $\bx$ in the sphere $\sphere :=\{ \bx\in\C^n\ |\ \|\bx\|_2 =1\}$.
For $a=2\mHC+1$, we denote the Taylor series for $Q(t)$ and $g(\bx;t)$ truncated after the $t^a/a!$ term with $Q_a(t)$ and $g_a(\bx;t)$, respectively.
We recall that $U_0 =I$ and $U_1 =-2\BH$.

\medskip
\makeblue{
Let the matrices $U_j$, $j\in\N_0$ denote the coefficients of $t^j/j!$ in the Taylor expansion~\eqref{Q-expans} and \eqref{A:Q}, such that
\begin{equation} \label{Uj}
 U_j := (-1)^j \sum_{k=0}^j \binom{j}{k} (B^*)^k B^{j-k} \,, \qquad j\in\N_0\, ,
\end{equation}
and note that each $U_j$ is self-adjoint. 
By Pascal's identity $\binom{j}{k} = \binom{j-1}{k-1} + \binom{j-1}{k}$ with the usual convention that $\binom{j}{-1} = \binom{j}{j+1} =0$,
\begin{align*}
 \sum_{k=0}^j \binom{j}{k} (B^*)^k B^{j-k}
&= \sum_{k=0}^j \left(  \binom{j-1}{k-1} (B^*)^k B^{j-k} + \binom{j-1}{k} (B^*)^k B^{j-k}\right)\\
&= \sum_{k=0}^j \left(  \binom{j-1}{k-1} (B^*)^{k-1}B^* B^{j-k} + \binom{j-1}{k} (B^*)^k B B^{j-1-k}\right)\\
&= \sum_{k=0}^{j-1}  \binom{j-1}{k} (B^*)^k (B^* + B)B^{j-1-k}
 = 2 \sum_{k=0}^{j-1} \binom{j-1}{k} (B^*)^k B_H B^{j-1-k} \,.
\end{align*}
Consequently,
\begin{equation} \label{Uj:2}
 U_j = (-1)^j\ 2 \sum_{k=0}^{j-1} \binom{j-1}{k} (B^*)^k B_H B^{j-1-k} \,, \qquad j\in\N\,.
\end{equation}
}

\medskip
First we outline the strategy of the proof, say for the case $\mHC=1$, i.e. $a=3$:
If $\bx\in\ker(U_1) =\ker(\BH)$ with $\|\bx\|_2=1$, then \eqref{x0} with $\bx_0=\bx$ holds.
\makeblue{Since $\mHC =1$, $B^*B_HB$ is positive definite on $\ker(B_H)$.
Consequently}, $\bx^* U_2 \bx=0$, $\bx^* U_3 \bx<0$, such that $g_a(\bx;t) =\tfrac{t^3}{3!} \bx^* U_3 \bx <0$ for $t>0$.
By contrast, if $\bx\notin\ker U_1$, we have $\bx^* U_1 \bx = -2\bx^* \BH\bx <0$.
Hence, for $\bx\notin\ker U_1$,
\[
 g_a(\bx;t)
 = -\hat{c} t +\bigO(t^2) \leq -\tilde{c} t^3
 \quad \text{for } t\to 0^+
\]
follows for some $\hat{c},\tilde{c}>0$ that depend on $\bx$.
Since $g_a(\bx;t)$ depends continuously on $\bx$, it is possible to combine these two estimates with a constant $c$ that is independent of $\bx\in\sphere$.
Since $(\ker U_1)^c \cap \sphere$ is not compact, we do not obtain a uniform estimate ``automatically''. 
So, the key aspect is here to obtain a uniform decay estimate for $\bx$ ``close to $\ker U_1$'', in the sense that $-\epsilon\leq \bx^* U_1 \bx\leq 0$.

\smallskip
\noindent
\underline{\textit{Step 1.}} 
\textit{Matrices with hypocoercivity index~$\mHC=1$.}
We suppose that matrix $B$ has hypocoercivity index $\mHC=1$, i.e. \makeblue{there exists $\kappa>0$ such that 
\begin{equation}\label{HC:kappa:mHC_1:app}
 \ttT_\mHC 
=\sum_{j=0}^{\mHC} (B^*)^j \BH B^j 
=\BH + B^* \BH B 
\geq \kappa I >0 \,. 
\end{equation}
}
Since $\mHC =1$, $B^*B_HB$ is positive definite on $\ker(B_H)$. 
Our goal is to estimate~$g(\bx;t)$ on $\sphere$.
\makeblue{
For $\bx\in\sphere$, we define
\begin{equation} \label{lambda_x:mu_x:mHC_1b}
 \lambda_\bx :=\ip{\bx}{\BH \bx}\geq 0\ , \quad
 \mu_\bx :=\ip{\bx}{B^* \BH B \bx}\geq 0\ , \quad
 \text{such that } \lambda_\bx +\mu_\bx \geq \kappa >0\ .
\end{equation}
}

\smallskip
\noindent
\underline{\textit{Step 1a.}} 
\makeviolet{Consider \fbox{$\bx\in\sphere$ with $\lambda_\bx \leq\makered{\delta}$} where $\delta\in(0,\kappa)$ will be chosen later and $\kappa>0$ such that~\eqref{HC:kappa:mHC_1:app} holds.
The key idea (to estimate $g(\bx;t)$ for $t\in[0,1]$)} is to collect the terms $t^j$ of order $j$ less than $a=2\mHC+1=3$ in a quadratic form which is non-positive.
Therefore, we use~\eqref{Uj:2} and Lemma~\ref{lm:sum-of-squares} with $U=(-B)^*$, $V=U_1$, $W=-B$ and \makeviolet{$m=\mHC-1=0$, to rewrite  $g(\bx;t)$ as}
\begin{equation*} 
\begin{split}
 g(\bx;t)
 = \bx^* \Big( \sum_{j=1}^{\infty} \frac{t^j}{j!}\ U_j \Big) \bx 
 &= \sum_{j=1}^{\infty} \frac{t^j}{j!} \sum_{k=0}^{j-1} \tbinom{j-1}{k} ((-B)^k \bx)^* U_1 (-B)^{j-k-1} \bx 
\\
 &= t \Bigg( \sum_{k=0}^{\infty} \frac{1}{(k+1)!} t^k ((-B)^k \bx)^* \Bigg) U_1
 \Bigg( \sum_{\ell=0}^{\infty} \frac{1}{(\ell+1)!} t^\ell (-B)^{\ell} \bx \Bigg) 
\\
 &\quad +\frac{t^3}{3!} \tbinom{2}{1} \tfrac14 ((-B) \bx)^* U_1 (-B) \bx 
 +\sum_{j=4}^{\infty} \frac{t^j}{j!} \sum_{k=1}^{j-2} \binom{j-1}{k} \Delta^{(1)}_{j,k} ((-B)^k \bx)^* U_1 (-B)^{j-k-1} \bx .
\end{split}
\end{equation*}
\makeviolet{
The first term is non-positive since $U_1 =-2\BH\leq 0$.
The second term is retained.
The third term can be estimated from above by $M_4\ t^4$ with
$M_4 := \sum_{j=4}^{\infty} \tfrac1{j!} (2\|B\|_2)^j$ by using \eqref{Uj:2}, \eqref{est:Uj}, and $\Delta^{(1)}_{j,k}<1$.
Altogether, we derive the estimate
}
\begin{equation} \label{est:g:C0:tilde}
 g(\bx;t) 
\leq \frac{1}{3!\ 2} \big((B\bx)^* U_1 B\bx\big)\ t^3 +M_4\ t^4 
 = \makeblue{-\frac{\mu_\bx}{3!} \ t^3 +M_4\ t^4 } .
\end{equation}

\medskip
To establish a uniform negative upper bound for $(B\bx)^* U_1 B\bx$ for $\bx\in\sphere$ with \makeblue{$\lambda_\bx \leq\makered{\delta}$}, we use \makeblue{
\eqref{lambda_x:mu_x:mHC_1b} to deduce $\mu_\bx\geq \kappa-\lambda_\bx \geq\kappa -\makered{\delta} >0$ \makered{since $\delta\in(0,\kappa)$}.
}
\makered{
For $\delta\in(0,\kappa)$, define
\begin{equation}
 \mu_\delta 
:= \min_{\bx\in\sphere \text{ with } \lambda_\bx \leq\delta} \mu_\bx
 = \min_{\bx\in\sphere \text{ with } \lambda_\bx \leq\delta} \ip{\bx}{B^* \BH B \bx}
\end{equation}
such that $\mu_\delta\geq \kappa-\delta>0$ and 
\begin{equation} \label{mu_delta:m_1}
 \mu_0
:=\lim_{\delta\to 0} \mu_\delta
 =\min_{\bx\in\ker\big(\tttT_0\big), \ \|\bx\|=1} \ip{\bx}{B^* \BH B \bx} .
\end{equation}
Then, we derive from~\eqref{est:g:C0:tilde} that 
\begin{equation} \label{est:g:C0:tilde:final:m_1}
 g(\bx;t)
 \leq -\frac{\makered{\mu_\delta}}{3!} t^3 +M_4\ t^4
 \qquad\text{ for all $\bx\in\sphere$ with $\lambda_\bx \leq\makered{\delta}$ and $t\in[0,1]$}.
\end{equation}
}

\smallskip
\noindent
\underline{\textit{Step 1b.}} 
\makeblue{Consider \fbox{$\bx\in\sphere$ with $\lambda_\bx >\makered{\delta}$} where \makered{$\delta\in(0,\kappa)$ will be chosen later and} $\kappa>0$ such that~\eqref{HC:kappa:mHC_1:app} holds.}
For $t\in[0,1]$, we deduce 
\begin{equation} \label{est:g:C0}
 g(\bx;t)
 := \bx^* \Big( \sum_{j=1}^{\infty} \frac{t^j}{j!}\ U_j \Big) \bx 
 = t \bx^* U_1 \bx +\sum_{j=2}^{\infty} \frac{t^j}{j!}\ \bx^* U_j \bx 
 \makeblue{ \le -2 \lambda_\bx t +M_2\ t^2} \,,
\end{equation}
since $\big| \sum_{j=2}^{\infty} \frac{t^j}{j!}\ \bx^* U_j \bx \big| \leq t^2 M_2$ with $M_2 :=\sum_{j=2}^{\infty} \frac{1}{j!}\ (2 \|B\|_2)^j$.
Then, \makeblue{$-2 \lambda_\bx t +M_2 t^2 \leq -\lambda_\bx t$ for all $0\leq t\leq {\lambda_\bx}/{M_2}$.}
For any given $c>0$, the estimate $-\lambda_\bx t \leq -c t^3$ holds if $0\leq t\leq \sqrt{ {\lambda_\bx}/{c}}$.
Define \makeblue{$t_\delta :=\min\{ {\makered{\delta}}/{ M_2}, \sqrt{ {\makered{\delta}}/{c}}, 1 \}$}.
Then, we derive
\begin{equation} \label{est:g:C0:final}
 g(\bx;t) \leq -c t^3 
\qquad \text{for all $\bx\in\sphere$ with \makeblue{$\lambda_\bx >\makered{\delta}$} and $t\in[0,t_\delta]$.}
\end{equation}

\smallskip
\makeviolet{
To sum up, choosing any $\delta\in(0,\kappa)$, the estimate~\eqref{est:g:C0:tilde:final:m_1} is derived.
Then, for $c:=\makered{\mu_\delta}/3!$, there exists a (sufficiently small) $t_\delta>0$ (as defined in \textit{Step 1b}) such that the estimate~\eqref{est:g:C0:final} holds.
Consequently, we obtain 
}
\makered{
\begin{equation} \label{est:g:m_1:mu_delta}
 g(\bx;t) 
\leq -\frac{\makered{\mu_\delta}}{3!} t^3 +M_4\ t^4 \quad \text{for all } \bx\in\sphere \text{ and } t\in[0,t_\delta].
\end{equation}
This shows~\eqref{est:Propagator} with $c_2 :=c/2 =\mu_\delta/(3!\ 2)$ and $a=2\mHC +1=3$.
}

\medskip
\noindent
\makeviolet{\underline{\textit{Step 1c.}}}
\makered{
To prove the second statement in Lemma~\ref{lm:Q:norm}, 
we improve the estimate of~$c$ as follows:
By definition, the time~$t_\delta$ satisfies $\lim_{\delta\to 0} t_\delta =0$.
To derive (the sharp) lower estimate~\eqref{c:lower_bound:mHC} on the multiplicative factor~$c$, 
we consider the Taylor expansion~\eqref{short-t-decay} of the propagator norm, use estimate~\eqref{est:g:m_1:mu_delta}, and take the limit $\delta\to 0$:
\begin{equation}
 -2c
=\lim_{\delta\to 0} \frac{\|e^{-B\ t_\delta}\|_2^2 -1}{t_\delta^3}
\makeviolet{=\lim_{\delta\to 0} \frac{g(\bx;t_\delta)}{t_\delta^3} }
\leq \lim_{\delta\to 0} \Big( -\frac{\mu_\delta}{3!} +\makeviolet{M_4\ t_\delta} \Big)
=-\frac{\mu_0}{3!} .
\end{equation}
Hence, we identified a lower estimate for the multiplicative factor~$c$ in~\eqref{short-t-decay} as 
\begin{equation}\label{c:lower_bound:m_1}
 c
\geq \frac{\mu_0}{3!\ 2} 
= \frac1{3!\ 2} \min_{\bx\in\ker\big(\tttT_0\big), \ \|\bx\|=1} \ip{\bx}{B^* \BH B \bx}.
\end{equation}
This finishes the proof of the second statement in Lemma~\ref{lm:Q:norm} in the case $\mHC=1$.
}

\bigskip
\noindent
\underline{\textit{Step 2.}} 
\textit{Matrices with hypocoercivity index~$\mHC\geq 2$.}
For matrices $B$ with hypocoercivity index~$\mHC\geq 2$,
\makeblue{
i.e. there exists $\kappa>0$ such that 
\begin{equation}\label{HC:kappa:mHC}
 \ttT_\mHC 
=\sum_{j=0}^{\mHC} (B^*)^j \BH B^j \geq \kappa I >0 \ , 
\end{equation}
}
we generalize this procedure as follows:
\makeblue{
We define, for $\bx\in\sphere$,
\begin{equation} \label{lambda_x:mu_x}
 \lambda_\bx :=\ip{\bx}{\ttT_{\mHC-1} \bx}\geq 0\ , \quad
 \mu_\bx :=\ip{\bx}{(B^*)^\mHC \BH B^\mHC \bx}\geq 0\ , \quad
 \text{such that } \lambda_\bx +\mu_\bx \geq \kappa >0\ .
\end{equation}
}

\noindent
\underline{\textit{Step 2a.}} 
\makered{Consider \fbox{$\bx\in\sphere$ with $\lambda_\bx \leq\delta$} where \makered{$\delta\in(0,\kappa)$ will be chosen later and} $\kappa>0$ such that~\eqref{HC:kappa:mHC} holds.}
For $t\in[0,1]$, we derive as in \textit{Step 2b} (see \eqref{id+est:g:Cl} below):
\begin{equation} \label{est:g:Cm}
 \begin{split}
g(\bx;t)
&\leq \frac{(B^{\mHC} \bx)^* U_1 B^{{\mHC}} \bx}{({2{\mHC}+1})! \binom{2{\mHC}}{\mHC}} t^{2{\mHC}+1}
 +M_{2{\mHC}+2}\ t^{2{\mHC}+2} 
 \\
&= -\frac{2}{({2{\mHC}+1})! \binom{2{\mHC}}{\mHC}} \mu_\bx t^{2{\mHC}+1}
 +M_{2{\mHC}+2}\ t^{2{\mHC}+2} \,,
 \end{split}
\end{equation}
with $M_{2{\mHC}+2} := \sum_{j=2{\mHC}+2}^{\infty} \tfrac1{j!} (2\|B\|_2)^j >0$.
\makeviolet{
To establish a uniform negative upper bound for $-\mu_\bx$ for $\bx\in\sphere$ with $\lambda_\bx \leq\makered{\delta}$, we use~\eqref{lambda_x:mu_x} to deduce  $\mu_\bx\geq \kappa-\lambda_\bx \geq\kappa-\makered{\delta} >0$
since $\delta\in(0,\kappa)$.
For $\delta\in(0,\kappa)$, define
\begin{equation}
 \mu_\delta 
:= \min_{\bx\in\sphere \text{ with } \lambda_\bx \leq\delta} \mu_\bx
 = \min_{\bx\in\sphere \text{ with } \lambda_\bx \leq\delta} \ip{\bx}{(B^*)^\mHC \BH B^\mHC \bx}
\end{equation}
such that $\mu_\delta\geq \kappa-\delta>0$ and 
\begin{equation} \label{mu_delta:mHC}
 \mu_0
:=\lim_{\delta\to 0} \mu_\delta
 =\min_{\bx\in\ker\big(\tttT_{\mHC-1}\big), \ \|\bx\|=1} \ip{\bx}{(B^*)^\mHC \BH B^\mHC \bx} .
\end{equation}
Then, we derive from~\eqref{est:g:Cm} that
\begin{subequations} \label{est:g:mHC:lambda_leq_delta}
\begin{equation} 
 g(\bx;t)
 \leq -c\ t^{2{\mHC}+1} +M_{2{\mHC}+2}\ t^{2{\mHC}+2}
 \qquad\text{ for all $\bx\in\sphere$ with $\lambda_\bx \leq\makered{\delta}$ and $t\in[0,1]$} ,
\end{equation}
where
\begin{equation} \label{est:g:mHC:lambda_leq_delta:c} 
 c:=\frac{2\mu_\delta}{({2{\mHC}+1})!\ \binom{2{\mHC}}{\mHC}} 
\ .
\end{equation}
\end{subequations}
}

\smallskip
\noindent
\underline{\textit{Step 2b.}} 
Next, we \makeblue{shall} show the following statement:
\makeblue{Consider \fbox{$\bx\in\sphere$ with $\lambda_\bx >\makered{\delta}$} where $\kappa>0$ such that~\eqref{HC:kappa:mHC}} holds.
For given $c>0$, there exists $t_\delta>0$ such that
\begin{equation} \label{est:g:II}
g(\bx;t) \leq -c\ t^{2\mHC +1} \qquad\text{for all $\bx\in\sphere$ with \makeblue{$\lambda_\bx > \makered{\delta}$} and $t\in[0,t_\delta]$.} 
\end{equation}

First, we decompose the sphere~$\sphere$ into the (non-disjoint) closed subsets
\begin{equation} \label{setsC}
\begin{split}
\setC_0 &:= \{ \bx\in\sphere\ |\ \bx^* U_1 \bx\leq -\epsilon \} \,, \\
\setC_1 &:= \{ \bx\in\sphere\ |\ -\epsilon \leq \bx^* U_1 \bx\ \land\ (B\bx)^* U_1 B\bx\leq -\epsilon \} \,, \\
\setC_2 &:= \{ \bx\in\sphere\ |\ -\epsilon \leq \bx^* U_1 \bx\ \land\ -\epsilon \leq (B\bx)^* U_1 B\bx\ 
\land\ (B^2 \bx)^* U_1 B^2 \bx\leq -\epsilon \} \,, \\
&\vdots \\
\setC_m &:= \{ \bx\in\sphere\ |\ \forall k\in\{0,\ldots,m-1\}:\ -\epsilon \leq (B^k \bx)^* U_1 B^k \bx\ 
\land\ (B^m \bx)^* U_1 B^m \bx\leq -\epsilon \} \,, \\
&\vdots \\
\setC_{\mHC-1} &:= \{ \bx\in\sphere\ |\ \forall k\in\{0,\ldots,\mHC-2\}:\ -\epsilon \leq (B^k \bx)^* U_1 B^k \bx\ \\ 
&\hspace{150pt}  \land\ (B^{\mHC-1} \bx)^* U_1 B^{\mHC-1} \bx\leq -\epsilon \} \,, 
\end{split}
\end{equation}
as well as $\setC_{\mHC} := \{ \bx\in\sphere\ |\ \forall k\in\{0,\ldots,\mHC-1\}:\ -\epsilon \leq (B^k \bx)^* U_1 B^k \bx\ \}$, for some positive parameter $\epsilon$ to be determined next.

We show that there exists $\epsilon >0$ such that 
$\setC_{\mHC} \subseteq \{ \bx\in\sphere\ |\ \makeblue{\lambda_\bx \leq\makered{\delta}} \}$:
Consider $\bx\in\setC_{\mHC}$.
Then, $\bx$ satisfies
\[
 \epsilon/2 \geq \ip{ \bx}{ (B^*)^k B_H B^k \bx} 
\quad\text{ for } k=0,\ldots,\mHC-1
\ ,
\]
which (upon summing up) implies that 
\[
 \mHC \epsilon/2
\geq \ip{ \bx}{ \sum_{k=0}^{\mHC-1} (B^*)^k B_H B^k\ \bx}
= \ip{ \bx}{ \ttT_{\mHC-1} \bx}
= \makeblue{\lambda_\bx}
\ .
\]
Choosing 
\begin{equation} \label{epsilon:alpha_0}
 \epsilon :=\frac{\makered{2\delta}}{ \mHC}
\quad\text{ implies }
 \setC_{\mHC} \subseteq \{ \bx\in\sphere\ |\ \makeblue{\lambda_\bx \leq\makered{\delta}} \} \ .
\end{equation}
Hence, the \makeblue{already established} estimate~\eqref{est:g:mHC:lambda_leq_delta} holds \makeblue{in particular} for $\bx\in\setC_{\mHC}$ and $t\in[0,\,1]$.
Using $\bigcup_{j=0}^{\mHC-1} \setC_j \supseteq \{\bx\in\sphere\ |\ \makeblue{\lambda_\bx > \makered{\delta}}\}$ \makeblue{(as the complementary inclusion of \eqref{epsilon:alpha_0})}, we are left to prove the estimate~\eqref{est:g:II} for all $\bx\in\bigcup_{j=0}^{\mHC-1} \setC_j$:

\smallskip
For all $\ell\in\{0,\ldots,\mHC-1\}$, \fbox{$\bx\in\setC_{\ell}$} and $t\in[0,1]$, the key idea (to estimate $g(\bx;t)$) is to collect the terms $t^j$ of order $j$ less than $2\ell +1$ in a quadratic form which is non-positive.
Therefore, we use again~\eqref{Uj:2} and Lemma~\ref{lm:sum-of-squares} with $U=-B^*$, $V=U_1$, $W=-B$ and $m=\ell-1$, to \makeviolet{rewrite $g(\bx;t)$ as}
\begin{subequations} \label{id+est:g:Cl}
\begin{equation} \label{id:g:Cl}
\begin{split}
g(\bx;t)
&= \bx^* \Big( \sum_{j=1}^{\infty} \frac{t^j}{j!}\ U_j \Big) \bx 
\\
&= \sum_{j=1}^{\infty} \frac{t^j}{j!} \sum_{k=0}^{j-1} \tbinom{j-1}{k} ((-B)^k \bx)^* U_1 (-B)^{j-k-1} \bx 
\\
&= \sum_{j=0}^{\ell-1} \frac{t^{2j+1}}{(2j+1)!} \frac1{\binom{2j}{j}}
    \Bigg( \sum_{k=0}^{\infty} \tfrac{(2j+1)!}{(k+2j+1)!} \tbinom{k+j}{j} t^k ((-B)^{k+j} \bx)^* \Bigg) U_1
    \Bigg( \sum_{k=0}^{\infty} \tfrac{(2j+1)!}{(k+2j+1)!} \tbinom{k+j}{j} t^k (-B)^{k+j} \bx \Bigg) 
\\
&\quad + \frac{t^{2\ell+1}}{({2\ell+1})!} \binom{2\ell}{\ell} \Delta^{(\ell)}_{{2\ell+1},\ell} ((-B)^\ell \bx)^* U_1 (-B)^{\ell} \bx 
\\
&\quad +\sum_{j=2\ell+2}^{\infty} \frac{t^j}{j!} \sum_{k=\ell}^{j-\ell-1} \binom{j-1}{k} \Delta^{(\ell)}_{j,k} ((-B)^k \bx)^* U_1 (-B)^{j-k-1} \bx .
\end{split}
\end{equation}
\makeviolet{
The first term is non-positive since $U_1 =-2\BH\leq 0$.
The second term is estimated using the assumption $\bx\in\setC_{\ell}$ and the identity $\Delta^{(\ell)}_{{2\ell+1},\ell} =\binom{2\ell}{\ell}^{-2}$.
The third term can be estimated from above by $M_{2\ell+2}\ t^{2\ell+2}$ with
$M_{2\ell+2} := \sum_{j=2\ell+2}^{\infty} \tfrac1{j!} (2\|B\|_2)^j >0$ by using \eqref{Uj:2}, \eqref{est:Uj}, and $\Delta^{(\ell)}_{j,k}\leq 1$ for $0\leq\ell\leq k\leq j-\ell-1$.
Altogether, we obtain the estimate
}
\begin{equation} \label{est:g:Cl}
g(\bx;t)
 \leq -\frac{\epsilon}{({2\ell+1})! \binom{2\ell}{\ell}} t^{2\ell+1} +M_{2\ell+2}\ t^{2\ell+2} .
\end{equation}
\end{subequations}

For given $c>0$ (e.g. as in~\eqref{est:g:mHC:lambda_leq_delta:c}), there exists $\tilde t_{\ell}>0$ \makered{(depending on $\epsilon$, with $\lim_{\epsilon\to 0} \tilde{t}_\ell =0$)} such that
\begin{equation} \label{est:g:Cl:final}
g(\bx;t) \leq -c\ t^{2\mHC +1} \qquad\text{for all $\bx\in\setC_{\ell}$ and $t\in[0,\tilde t_\ell]$.}  
\end{equation}
Choosing $\epsilon =\makered{2\delta/\mHC}$ as in~\eqref{epsilon:alpha_0} (such that $\bigcup_{j=0}^{\mHC-1} \setC_j \supseteq \{\bx\in\sphere\ |\ \makeblue{\lambda_\bx > \makered{\delta}} \}$) and
\begin{equation} \label{est:g:II:t_0}
 t_\delta := \min\{\tilde  t_\ell\ , \ \ell=0,\ldots,\mHC-1 \} \ ,
\end{equation}
implies the estimate~\eqref{est:g:II}.

\medskip 
\makeblue{
To sum up, for \makered{any} fixed $\delta\in(0,\kappa)$,
the estimate~\eqref{est:g:mHC:lambda_leq_delta} with multiplicative constant~$c$ in~\eqref{est:g:mHC:lambda_leq_delta:c} is proven in~\textit{Step~2a}.
Then, for $\epsilon =\makered{2\delta}/\mHC$ and $c$ in~\eqref{est:g:mHC:lambda_leq_delta:c}, there exists a (sufficiently small) $t_\delta>0$ \makered{(as defined in~\eqref{est:g:II:t_0})} such that the estimate~\eqref{est:g:II} holds.
\makeviolet{
Consequently, we obtain 
\begin{equation} \label{est:g:mHC:mu_delta}
 g(\bx;t)
 \leq -\frac{2\mu_\delta}{({2{\mHC}+1})! \binom{2{\mHC}}{\mHC}} t^{2{\mHC}+1} +M_{2{\mHC}+2}\ t^{2{\mHC}+2}
 \qquad\text{ for all $\bx\in\sphere$ with $t\in[0,t_\delta]$.}
\end{equation}
}
This shows~\eqref{est:Propagator} with $c_2:=c/2$, $c$ as defined in~\eqref{est:g:mHC:lambda_leq_delta:c} and $a=2\mHC +1$.
This finishes the proof of the first statement in Lemma~\ref{est:Propagator} for $\mHC\in\N$.
}

\medskip
\noindent
\makeviolet{\underline{\textit{Step 2c.}}}
\makered{
To prove the second statement in Lemma~\ref{lm:Q:norm}, we improve the estimate of~\makeviolet{$c$} as follows:
By definition, the time~$t_\delta$ depends on $\delta$ (since $\epsilon=2\delta/\mHC$) such that $\lim_{\delta\to 0} t_\delta =0$. 
To derive the (sharp) lower estimate~\eqref{c:lower_bound:mHC} on the multiplicative factor~$c$, we consider the Taylor expansion~\eqref{short-t-decay} of the propagator norm, use estimate~\eqref{est:g:mHC:mu_delta}, and take the limit $\delta\to 0$ such that 
\begin{equation}
 -2c
=\lim_{\delta\to 0} \frac{\|e^{-B\ t_\delta}\|_2^2 -1}{t_\delta^{2\mHC+1}}
\leq \lim_{\delta\to 0} \Bigg(-\frac{2\mu_\delta}{({2{\mHC}+1})! \binom{2{\mHC}}{\mHC}} +M_{2{\mHC}+2}\ t_\delta \Bigg)
=-\frac{2\mu_0}{({2{\mHC}+1})! \binom{2{\mHC}}{\mHC}} .
\end{equation}
This proves the lower estimate~\eqref{c:lower_bound:mHC} for the multiplicative factor~$c$ in~\eqref{short-t-decay}.
}
\end{proof}

The proof of Lemma~\ref{lm:Q:norm} uses the following identity:
\begin{lm} \label{lm:sum-of-squares}
Let $U, V, W\in\Cnn$. For all $m\in\N_0$, the following identity holds
\begin{equation} \label{id:sum-of-squares}
 \begin{split}
&\sum_{j=1}^{\infty} \frac{t^j}{j!} \sum_{k=0}^{j-1} \tbinom{j-1}{k} U^k V W^{j-k-1}
\\
&= \sum_{j=0}^m \frac{t^{2j+1}}{(2j+1)!} \frac1{\binom{2j}{j}}
       \Bigg( \sum_{k=0}^{\infty} \frac{(2j+1)!}{(k+2j+1)!} \binom{k+j}{j} t^k U^{k+j} \Bigg) V
       \Bigg( \sum_{\ell=0}^{\infty} \frac{(2j+1)!}{(\ell+2j+1)!} \binom{\ell+j}{j} t^\ell W^{\ell+j} \Bigg) 
\\
&\quad +\sum_{j=2m+3}^{\infty} \frac{t^j}{j!} \sum_{k=m+1}^{j-m-2} \binom{j-1}{k} \Delta^{(m+1)}_{j,k} U^k V W^{j-k-1} \,,
 \end{split}
\end{equation}
where $\Delta^{(m)}_{j,k} :=\frac{\binom{k}{m} \binom{j-k-1}{m}}{\binom{k+m}{m} \binom{j-k-1+m}{m}}$ for all \makeblue{$m\leq k$ and $m\leq j-k-1$}.
\end{lm}
\begin{proof}
%
We will prove the identity by induction.
For $m=0$, we have to prove the identity
\begin{multline*}
\sum_{j=1}^{\infty} \frac{t^j}{j!} \sum_{k=0}^{j-1} \tbinom{j-1}{k} U^k V W^{j-k-1} \\
= t \Bigg( \sum_{k=0}^{\infty} \tfrac{1}{(k+1)!} t^k U^k \Bigg) V 
\Bigg( \sum_{\ell=0}^{\infty} \tfrac{1}{(\ell+1)!} t^\ell W^\ell \Bigg) 
+\sum_{j=3}^{\infty} \frac{t^j}{j!} \sum_{k=1}^{j-2} \tbinom{j-1}{k} \tfrac{k (j-k-1)}{(k+1) (j-k)} U^k V W^{j-k-1}
\end{multline*} 
since $\Delta^{(1)}_{j,k} =\frac{\binom{k}{1} \binom{j-k-1}{1}}{\binom{k+1}{1} \binom{j-k}{1}} =\tfrac{k (j-k-1)}{(k+1) (j-k)}$.
The first term on the right hand side can be written by the Cauchy product formula as
\begin{align*}
& t \Bigg( \sum_{k=0}^{\infty} \tfrac{1}{(k+1)!} t^k U^k \Bigg) V 
\Bigg( \sum_{\ell=0}^{\infty} \tfrac{1}{(\ell+1)!} t^\ell W^\ell \Bigg) 
\\
&= t \sum_{j=0}^{\infty} t^j \sum_{k=0}^{j} \tfrac{1}{(k+1)!} U^k V \tfrac{1}{(j-k+1)!} W^{j-k} 
= \sum_{j=0}^{\infty} t^{j+1} \sum_{k=0}^{j} \tfrac{1}{(k+1)!} U^k V \tfrac{1}{(j-k+1)!} W^{j-k} 
\\
&= \sum_{j=1}^{\infty} \frac{t^j}{j!} \sum_{k=0}^{j-1} \tfrac{j!}{(k+1)! (j-k)!} U^k V W^{j-k-1} 
= \sum_{j=1}^{\infty} \frac{t^j}{j!} \sum_{k=0}^{j-1} \tbinom{j-1}{k} \tfrac{j}{(k+1) (j-k)} U^k V W^{j-k-1} 
\\
&= \sum_{j=1}^{2} \frac{t^j}{j!} \sum_{k=0}^{j-1} \tbinom{j-1}{k} U^k V W^{j-k-1} 
+\sum_{j=3}^{\infty} \frac{t^j}{j!} \sum_{k=0}^{j-1} \tbinom{j-1}{k} \tfrac{j}{(k+1) (j-k)} U^k V W^{j-k-1} \,.
\end{align*}
Therefore,
\begin{align*}
&t \Bigg( \sum_{k=0}^{\infty} \tfrac{1}{(k+1)!} t^k U^k \Bigg) V 
       \Bigg( \sum_{\ell=0}^{\infty} \tfrac{1}{(\ell+1)!} t^\ell W^\ell \Bigg) 
     +\sum_{j=3}^{\infty} \frac{t^j}{j!} \sum_{k=1}^{j-2} \tbinom{j-1}{k} \tfrac{k (j-k-1)}{(k+1) (j-k)} U^k V W^{j-k-1} 
\\
&= \sum_{j=1}^{2} \frac{t^j}{j!} \sum_{k=0}^{j-1} \tbinom{j-1}{k} U^k V W^{j-k-1} 
     +\sum_{j=3}^{\infty} \frac{t^j}{j!} \sum_{k=0}^{j-1} \tbinom{j-1}{k} \tfrac{j}{(k+1) (j-k)} U^k V W^{j-k-1}
     +\sum_{j=3}^{\infty} \frac{t^j}{j!} \sum_{k=1}^{j-2} \tbinom{j-1}{k} \tfrac{k (j-k-1)}{(k+1) (j-k)} U^k V W^{j-k-1} 
\\     
&= \sum_{j=1}^{\infty} \frac{t^j}{j!} \sum_{k=0}^{j-1} \tbinom{j-1}{k} U^k V W^{j-k-1} \,.
\end{align*} 
We assume that the formula holds for $m\in\N_0$ and prove it for $m+1$.
First, we use again the Cauchy product formula to derive
\begin{align*}
&\frac{t^{2m+3}}{(2m+3)!} \frac1{\binom{2m+2}{m+1}}
   \Bigg( \sum_{k=0}^{\infty} \tfrac{(2m+3)!}{(k+2m+3)!} \tbinom{k+m+1}{m+1} t^k U^{k+m+1} \Bigg) V
   \Bigg( \sum_{\ell=0}^{\infty} \tfrac{(2m+3)!}{(\ell+2m+3)!} \tbinom{\ell+m+1}{m+1} t^\ell W^{\ell+m+1} \Bigg) 
\\
&= \frac{t^{2m+3}}{(2m+3)!} \frac1{\binom{2m+2}{m+1}} 
     \sum_{j=0}^{\infty} t^j \sum_{k=0}^{j} \tfrac{(2m+3)!}{(k+2m+3)!} \tbinom{k+m+1}{m+1} U^{k+m+1} V \tfrac{(2m+3)!}{(j-k+2m+3)!} \tbinom{j-k+m+1}{m+1} W^{j-k+m+1} 
\\
&= \sum_{j=0}^{\infty} \frac{t^{j+2m+3}}{(j+2m+3)!} \sum_{k=0}^{j} \tfrac{(j+2m+3)!}{(2m+3)! \binom{2m+2}{m+1}} \tfrac{(2m+3)!}{(k+2m+3)!} \tbinom{k+m+1}{m+1} U^{k+m+1} V \tfrac{(2m+3)!}{(j-k+2m+3)!} \tbinom{j-k+m+1}{m+1} W^{j-k+m+1} 
\\
&= \sum_{j=2m+3}^{\infty} \frac{t^j}{j!} \sum_{k=0}^{j-2m-3} \tfrac{j!}{(2m+3)! \binom{2m+2}{m+1}} \tfrac{(2m+3)!}{(k+2m+3)!} \tbinom{k+m+1}{m+1} U^{k+m+1} V \tfrac{(2m+3)!}{(j-k)!} \tbinom{j-k-m-2}{m+1} W^{j-k-m-2} 
\\
&= \sum_{j=2m+3}^{\infty} \frac{t^j}{j!} \sum_{k=m+1}^{j-m-2} \tfrac{j!}{(2m+3)! \binom{2m+2}{m+1}} \tfrac{(2m+3)!}{(k+m+2)!} \tbinom{k}{m+1} \tfrac{(2m+3)!}{(j-k+m+1)!} \tbinom{j-k-1}{m+1} U^{k} V W^{j-k-1} \,.
\end{align*}
Therefore,
\begin{align}
&\sum_{j=0}^{m+1} \frac{t^{2j+1}}{(2j+1)!} \frac1{\binom{2j}{j}}
       \Bigg( \sum_{k=0}^{\infty} \frac{(2j+1)!}{(k+2j+1)!} \binom{k+j}{j} t^k U^{k+j} \Bigg) V
       \Bigg( \sum_{\ell=0}^{\infty} \frac{(2j+1)!}{(\ell+2j+1)!} \binom{\ell+j}{j} t^\ell W^{\ell+j} \Bigg) \nonumber 
\\
&\quad +\sum_{j=2m+5}^{\infty} \frac{t^j}{j!} \sum_{k=m+2}^{j-m-3} \binom{j-1}{k} \Delta^{(m+2)}_{j,k} U^k V W^{j-k-1} \nonumber 
\\
&=\sum_{j=0}^{m} \frac{t^{2j+1}}{(2j+1)!} \frac1{\binom{2j}{j}}
  \Bigg( \sum_{k=0}^{\infty} \frac{(2j+1)!}{(k+2j+1)!} \binom{k+j}{j} t^k U^{k+j} \Bigg) V
  \Bigg( \sum_{\ell=0}^{\infty} \frac{(2j+1)!}{(\ell+2j+1)!} \binom{\ell+j}{j} t^\ell W^{\ell+j} \Bigg) \nonumber 
\\
&\quad +\sum_{j=2m+3}^{\infty} \frac{t^j}{j!} \sum_{k=m+1}^{j-m-2} \tfrac{j!}{(2m+3)! \binom{2m+2}{m+1}} \tfrac{(2m+3)!}{(k+m+2)!} \tbinom{k}{m+1} \tfrac{(2m+3)!}{(j-k+m+1)!} \tbinom{j-k-1}{m+1} U^{k} V W^{j-k-1} \nonumber 
\\
&\quad +\sum_{j=2m+5}^{\infty} \frac{t^j}{j!} \sum_{k=m+2}^{j-m-3} \binom{j-1}{k} \Delta^{(m+2)}_{j,k} U^k V W^{j-k-1} \,. \label{last-id}
\end{align}  
Using $\Delta^{(m+2)}_{j,k} =\Delta^{(m+1)}_{j,k} \frac{(k-m-1)(j-k-m-2)}{(k+m+2)(j-k+m+1)}$ we deduce that~\eqref{last-id} equals:
\begin{align*}
  &=\sum_{j=0}^{m} \frac{t^{2j+1}}{(2j+1)!} \frac1{\binom{2j}{j}}
       \Bigg( \sum_{k=0}^{\infty} \frac{(2j+1)!}{(k+2j+1)!} \binom{k+j}{j} t^k U^{k+j} \Bigg) V
       \Bigg( \sum_{\ell=0}^{\infty} \frac{(2j+1)!}{(\ell+2j+1)!} \binom{\ell+j}{j} t^\ell W^{\ell+j} \Bigg) 
\\
  &\quad +\sum_{j=2m+3}^{\infty} \frac{t^j}{j!} \sum_{k=m+1}^{j-m-2} \tbinom{j-1}{k} \Delta^{(m+1)}_{j,k} \tfrac{(2m+3) j}{(k+m+2) (j-k+m+1)} U^{k} V W^{j-k-1} 
\\
  &\quad +\sum_{j=2m+5}^{\infty} \frac{t^j}{j!} \sum_{k=m+2}^{j-m-3} \tbinom{j-1}{k} \Delta^{(m+1)}_{j,k} \frac{(k-m-1)(j-k-m-2)}{(k+m+2)(j-k+m+1)} U^k V W^{j-k-1} 
\\
  &= \sum_{j=0}^m \frac{t^{2j+1}}{(2j+1)!} \frac1{\binom{2j}{j}}
       \Bigg( \sum_{k=0}^{\infty} \frac{(2j+1)!}{(k+2j+1)!} \binom{k+j}{j} t^k U^{k+j} \Bigg) V
       \Bigg( \sum_{\ell=0}^{\infty} \frac{(2j+1)!}{(\ell+2j+1)!} \binom{\ell+j}{j} t^\ell W^{\ell+j} \Bigg) 
\\
  &\quad +\sum_{j=2m+3}^{\infty} \frac{t^j}{j!} \sum_{k=m+1}^{j-m-2} \binom{j-1}{k} \Delta^{(m+1)}_{j,k} U^k V W^{j-k-1} 
\\
  &=\sum_{j=1}^{\infty} \frac{t^j}{j!} \sum_{k=0}^{j-1} \tbinom{j-1}{k} U^k V W^{j-k-1} \,,
\end{align*}
where we used the induction hypothesis, i.e.~\eqref{id:sum-of-squares}, in the final equality.
This finishes the proof.
\end{proof}

\makered{
To determine the optimal multiplicative factor~$c$ in~\eqref{short-t-decay} (see Theorem~\ref{th:HC-decay}\ref{th:HC-decay:b}), we shall next derive an improved upper estimate for~$c$, compared to~\eqref{c:lower_bound}:
}

\makered{
\begin{lm}\label{lm:c:upper_bound}
Let the ODE system~\eqref{ODE:B} be conservative-dissipative, and let the system matrix $B$ be (hypo)coercive with hypocoercivity index $\mHC\in\N$.
Then, the multiplicative factor~$c$ in~\eqref{short-t-decay} satisfies
\begin{equation}\label{c:upper_bound:mHC} 
 c\leq c_1
\qquad\text{with }
 c_1
:=\frac1{(2\mHC+1)! \binom{2\mHC}{\mHC}} \min_{\bx_0\in  \ker\big(\tttT_{\mHC-1}\big), \ \|\bx_0\|_2=1} \ip{B^\mHC \bx_0}{\BH B^\mHC \bx_0}.
\end{equation}
\end{lm}
\begin{proof}
As illustrated by Example~\ref{ex:envelope}, see also Figure~\ref{fig:envelope}, the propagator norm is in general not determined by the norm of one specific solution.
Instead we consider a parameterized family of solutions pertaining to initial values~$\bx_\tau$, $\tau\in[0,1]$:

\smallskip
Due to Lemma~\ref{lm:g} below, for $\bx_0\in\ker\big(\ttT_{\mHC-1}\big)$ (with $\|\bx_0\|_2=1$), there exist real constants $b_\ell$, $\ell=1,\ldots,\mHC$ such that 
\[ \bx_\tau :=\bx_0 +\sum_{\ell=1}^{\mHC} b_\ell \tau^\ell B^\ell \bx_0 \]
satisfies
\begin{equation} \label{g:x_tau:tau}
 g(\bx_\tau;\tau)
:= \bx_\tau^* \Big( \sum_{j=1}^{\infty} \frac{\tau^j}{j!}\ U_j \Big) \bx_\tau 
=-2c_1(\bx_0) \tau^{2\mHC+1} +\cO(\tau^{2\mHC+2}) , 
 \quad\text{for } \tau\in[0,1] ,
\end{equation}
with $c_1(\bx_0)$ defined in~\eqref{c_1:x_0} and $\lim_{\tau\to 0} \bx_\tau=\bx_0$.
To normalize the family of vectors $\bx_\tau$, $\tau\in[0,1]$ we define
\begin{equation}
 \widetilde{\bx}_\tau := \frac{\bx_\tau}{\|\bx_\tau\|_2} ,
\end{equation}
still satisfying $\lim_{\tau\to 0} \widetilde{\bx}_\tau=\bx_0$.
For $\tau\in[0,1]$, we estimate the propagator norm as
\begin{equation} \label{est:Q:norm:lower_bound_t}
 \big\|e^{-B \tau}\big\|_2^2 
\geq \big\|e^{-B t} \widetilde{\bx}_\tau\big\|_2^2 \Big|_{t=\tau}
=\frac{\big\|e^{-B \tau} \bx_\tau\big\|_2^2 }{\|\bx_\tau\|_2^2}
=\frac{\|\bx_\tau\|_2^2 +g(\bx_\tau;\tau)}{\|\bx_\tau\|_2^2}
=1 -2\frac{c_1(\bx_0)}{\|\bx_\tau\|_2^2} \tau^{2\mHC+1} +\frac{\cO(\tau^{2\mHC+2})}{\|\bx_\tau\|_2^2} ,
\end{equation}
where we used definition~\eqref{est:g:Cm} and~\eqref{g:x_tau:tau}.
To derive a bound for the multiplicative factor~$c$ in~\eqref{short-t-decay}, we consider the Taylor expansion~\eqref{short-t-decay} of the propagator norm, use estimate~\eqref{est:Q:norm:lower_bound_t} for $\tau>0$, and take the limit $\tau\to 0$ such that 
\[
 -2 c 
=\lim_{\tau\to 0} \frac{\big\|e^{-B \tau}\big\|_2^2 -1}{\tau^{2\mHC+1}} 
\geq  \lim_{\tau\to 0} \frac{g(\bx_\tau;\tau)}{\|\bx_\tau\|_2^2\ \tau^{2\mHC+1}} 
=\lim_{\tau\to 0} \bigg(-2\frac{c_1(\bx_0)}{\|\bx_\tau\|_2^2} +\frac{\cO(\tau)}{\|\bx_\tau\|_2^2} \bigg)
=-2c_1(\bx_0) .
\]
\makeviolet{Hence, $c\leq c_1(\bx_0)$ for all normalized vectors $\bx_0\in\ker\big(\ttT_{\mHC-1}\big)$.
Taking the minimum of $c_1(\bx_0)$ over all normalized vectors $\bx_0\in\ker\big(\ttT_{\mHC-1}\big)$ yields the upper bound for the multiplicative factor~$c$ as given in~\eqref{c:upper_bound:mHC}. }
This finishes the proof.
\end{proof}
}

\makered{The proof of Lemma~\ref{lm:c:upper_bound} uses the following construction of a vector function:}
\makered{
\begin{lm}\label{lm:g}
Let the ODE system~\eqref{ODE:B} be conservative-dissipative, and let the system matrix $B$ be hypocoercive with hypocoercivity index $\mHC\in\N$.
Then, for each $\bx_0\in\ker\big(\ttT_{\mHC-1}\big)$, there exists a polynomial vector function $\bx_\tau\in\Cn$, $\tau\in[0,1]$ of the form 
\begin{equation} \label{x_ast}
 \bx_\tau =\bx_0 +\sum_{\ell=1}^{\mHC} b_\ell \tau^\ell B^\ell \bx_0 ,
\qquad \text{for a suitable choice of } b_\ell\in\R , \quad \ell=1,\ldots,\mHC ,
\end{equation}
such that 
\begin{equation} \label{SolutionNorm:expansion}
g(\bx_\tau;\tau)
:= \bx_\tau^* \Big( \sum_{j=1}^{\infty} \frac{\tau^j}{j!}\ U_j \Big) \bx_\tau 
= -2 c_1(\bx_0) \tau^a +\bigO(\tau^{a+1}) 
\qquad \text{for } \tau\in[0,1] ,
\end{equation}
where $a =2 \mHC +1$ and
\begin{equation} \label{c_1:x_0}
 c_1(\bx_0)
:= \frac{\big\|\sqrt{\BH} B^\mHC \bx_0\big\|_2^2}{(2\mHC+1)!\ \binom{2\mHC}{\mHC}} .
\end{equation}
\end{lm}
\begin{proof}
For $\bx\in\Cn$ and $\tau\in[0,1]$, consider $g(\bx;\tau) := \bx^* \Big( \sum_{j=1}^{\infty} \frac{\tau^j}{j!}\ U_j \Big) \bx$ using $U_j$ in the form~\eqref{Uj:2}.
Following Lemma~\ref{lm:sum-of-squares} with $U=-B^*$, $V=U_1$, $W=-B$, and $m=\mHC-1$, we rewrite $g(\bx;\tau)$ as
\begin{equation} \label{est:g:mHC}
\begin{split}
g(\bx;\tau)
&= \sum_{j=1}^{\infty} \frac{\tau^j}{j!} \sum_{k=0}^{j-1} \tbinom{j-1}{k} ((-B)^k \bx)^* U_1 (-B)^{j-k-1} \bx 
\\
&= \sum_{j=0}^{\mHC-1} \frac{\tau^{2j+1}}{(2j+1)!} \frac1{\binom{2j}{j}}
    \Bigg( \sum_{k=0}^{\infty} \tfrac{(2j+1)!}{(k+2j+1)!} \tbinom{k+j}{j} \tau^k ((-B)^{k+j} \bx)^* \Bigg) U_1
    \Bigg( \sum_{k=0}^{\infty} \tfrac{(2j+1)!}{(k+2j+1)!} \tbinom{k+j}{j} \tau^k (-B)^{k+j} \bx \Bigg) 
\\
&\quad + \frac{\tau^{2\mHC+1}}{({2\mHC+1})!} \binom{2\mHC}{\mHC} \Delta^{(\mHC)}_{{2\mHC+1},\mHC} ((-B)^\mHC \bx)^* U_1 (-B)^{\mHC} \bx 
\\
&\quad +\sum_{j=2\mHC+2}^{\infty} \frac{\tau^j}{j!} \sum_{k=\mHC}^{j-\mHC-1} \binom{j-1}{k} \Delta^{(\mHC)}_{j,k} ((-B)^k \bx)^* U_1 (-B)^{j-k-1} \bx 
\\
&=-2 \sum_{j=0}^{\mHC-1} \frac{\tau}{(2j+1)!} \frac1{\binom{2j}{j}}
    \Bigg\| \sqrt{\BH}
    \Bigg( \sum_{k=0}^{\infty} \tfrac{(2j+1)!}{(k+2j+1)!} \tbinom{k+j}{j} \tau^{k+j} (-B)^{k+j} \bx \Bigg) \Bigg\|_2^2 
\\
&\quad -2 \frac{\tau^{2\mHC+1}}{({2\mHC+1})!\ \binom{2\mHC}{\mHC} } \big\| \sqrt{\BH} B^{\mHC} \bx \big\|_2^2 
 +\bigO(\tau^{2\mHC+2}) \,,
\end{split}
\end{equation}
using that $U_1 =-2\BH$ is a negative semi-definite Hermitian matrix, and the identity 
\[ 
 \Delta^{(\mHC)}_{{2\mHC+1},\mHC} =\binom{2\mHC}{\mHC}^{-2} ,
\]
to rewrite the first and second term, respectively.
The third term in~\eqref{est:g:mHC} can be bounded by $M_{2\mHC+2}\ \tau^{2\mHC+2}$ with $M_{2\mHC+2} := \sum_{j=2\mHC+2}^{\infty} \tfrac1{j!} (2\|B\|_2)^j >0$, using that $\Delta^{(\mHC)}_{j,k}\leq 1$ for $0\leq\mHC\leq k\leq j-\mHC-1$ and $\tau\in[0,1]$.

\medskip
\noindent
\underline{\textit{Step 1.}} 
To estimate the second term in the last identity of~\eqref{est:g:mHC} for $\bx=\bx_\tau$, we use a polynomial ansatz for $\bx_\tau$:

For $\tau\in[0,1]$, we consider the ansatz
\begin{equation} \label{ansatz:x}
 \bx_\tau :=\sum_{\ell=0}^{\mHC} \tau^\ell \bx_\ell
\qquad\text{with the given } \bx_0\in\ker\big(\ttT_{\mHC-1}\big)
\text{ and some $\bx_\ell\in\Cn$, $\ell=1,\ldots,\mHC$ to be chosen.}
\end{equation}
Then, we observe that
\begin{equation} \label{est:g:mHC:secondTerm}
 \big\| \sqrt{\BH} B^{\mHC} \bx_\tau\big\|_2^2
=\big\| \sqrt{\BH} B^{\mHC} \sum_{\ell=0}^{\mHC} \tau^\ell \bx_\ell\big\|_2^2
=\big\| \sqrt{\BH} B^{\mHC} \bx_0\big\|_2^2 +\cO(\tau) ,
\end{equation}
such that the second term in the last identity of~\eqref{est:g:mHC} satisfies
\begin{equation} \label{est:g:mHC:2nd}
 -2 \frac{\tau^{2\mHC+1}}{({2\mHC+1})!\ \binom{2\mHC}{\mHC} } \big\| \sqrt{\BH} B^{\mHC} \bx_\tau \big\|_2^2
= -2 \frac{\tau^{2\mHC+1}}{({2\mHC+1})!\ \binom{2\mHC}{\mHC} } \big\| \sqrt{\BH} B^{\mHC} \bx_0\big\|_2^2 +\cO(\tau^{2\mHC+2}) .
\end{equation}

\medskip
\noindent
\underline{\textit{Step 2.}} 
To estimate the first term in the last identity of~\eqref{est:g:mHC} for $\bx=\bx_\tau$, we refine the ansatz~\eqref{ansatz:x} for $\bx_\tau$ as follows:

Consider~\eqref{ansatz:x} with
\begin{equation} \label{ansatz:x_ell}
 \bx_\ell :=b_\ell B^\ell \bx_0 , 
\quad\text{with some } b_\ell\in\R , 
\quad \ell=1,\ldots,\mHC \text{ to be chosen.}
\end{equation}
We shall construct the coefficients $b_\ell\in\R$, $\ell=1,\ldots,\mHC$ and set $b_0:=1$ such that the first term in the last identity of~\eqref{est:g:mHC} satisfies
\begin{equation} \label{est:g:mHC:1st}
 \sum_{j=0}^{\mHC-1} \frac{\tau}{(2j+1)!} \frac1{\binom{2j}{j}}
    \Bigg\| \sqrt{\BH} \Bigg( \sum_{k=0}^{\infty} \tfrac{(2j+1)!}{(k+2j+1)!} \tbinom{k+j}{j} \tau^{k+j} (-B)^{k+j} \bx_\tau \Bigg) \Bigg\|_2^2 
=\cO(\tau^{2\mHC+2}) .
\end{equation}
Each term in the outer sum is non-negative.
Therefore, for $j=0,\ldots,\mHC-1$, we consider each
\begin{equation} \label{S_j}
 S_j
:=\Bigg\| \sqrt{\BH} \Bigg( \sum_{k=0}^{\infty} \tfrac{(2j+1)!}{(k+2j+1)!} \tbinom{k+j}{j} \tau^{k+j} (-B)^{k+j} \sum_{\ell=0}^{\mHC} b_\ell \tau^\ell B^\ell \bx_0 \Bigg) \Bigg\|_2^2 
\end{equation}
separately, and construct $b_\ell$, $\ell=1,\ldots,\mHC$ iteratively such that $S_j=\cO(\tau^{2\mHC+1})$.


\medskip
Starting with $\ell=1$, we determine $b_\ell=b_1$ by considering $S_{\mHC-\ell}=S_{\mHC-1}$:
Using
\begin{equation} \label{x0:m-1} 
 \sqrt\BH \bx_0 =\ldots =\sqrt\BH B^{\mHC-1}\bx_0 =0 , 
\end{equation}
we can rewrite $S_{\mHC-1}$ as
\begin{equation} \label{S_mHC-1}
\begin{split}
 S_{\mHC-1}
&= \Bigg\| \sqrt{\BH} \Bigg( \sum_{k=0}^{\infty} \tfrac{(2\mHC-1)!}{(k+2\mHC-1)!} \tbinom{k+\mHC-1}{\mHC-1} \tau^{k+\mHC-1} (-B)^{k+\mHC-1} \sum_{\ell=0}^{\mHC} b_\ell \tau^\ell B^\ell \bx_0 \Bigg) \Bigg\|_2^2 
\\
&= \Big\| \sqrt{\BH} \Big( 
\tau^{\mHC-1} (-B)^{\mHC-1} b_1 \tau B \bx_0 
+\tfrac{(2\mHC-1)!}{(2\mHC)!} \tbinom{\mHC}{\mHC-1} \tau^{\mHC} (-B)^{\mHC} \bx_0 
+\cO(\tau^{\mHC+1})
\Big) \Big\|_2^2 
\\
&= \tau^{2\mHC} \Big\| \sqrt{\BH} \Big( b_1 
-\tfrac{(2\mHC-1)!}{(2\mHC)!} \tbinom{\mHC}{\mHC-1} \Big) (-B)^\mHC \bx_0 
\Big\|_2^2 +\cO\big(\tau^{2\mHC+1}\big) .
\end{split}
\end{equation}
Choosing
\begin{equation}\label{b_1}
 b_1
:=\tfrac{(2\mHC-1)!}{(2\mHC)!} \tbinom{\mHC}{\mHC-1}
 =\tfrac12
\end{equation}
yields $S_{\mHC-1}=\cO(\tau^{2\mHC+1})$.

\smallskip
Subsequently, for $\ell=2,\ldots,\mHC$, we determine $b_\ell$ by considering $S_{\mHC-\ell}$ and using~\eqref{x0:m-1}:
\begin{equation*} 
\begin{split}
 S_{\mHC-\ell}
&=\Bigg\| \sqrt{\BH} \Bigg( \sum_{k=0}^{\infty} \tfrac{(2\mHC-2\ell+1)!}{(k+2\mHC-2\ell+1)!} \tbinom{k+\mHC-\ell}{\mHC-\ell} \tau^{k+\mHC-\ell} (-B)^{k+\mHC-\ell} \sum_{p=0}^{\mHC} b_p \tau^p B^p \bx_0 \Bigg) \Bigg\|_2^2 
\\
&=\Bigg\| \sqrt{\BH} \Bigg( \sum_{k=0}^{\ell} \tfrac{(2\mHC-2\ell+1)!}{(k+2\mHC-2\ell+1)!} \tbinom{k+\mHC-\ell}{\mHC-\ell} \tau^{k+\mHC-\ell} (-B)^{k+\mHC-\ell} b_{\ell-k} \tau^{\ell-k} B^{\ell-k} \bx_0 \Bigg) \Bigg\|_2^2
+\cO(\tau^{2\mHC+1})
\\
&=\tau^{2\mHC} \Bigg\| \sqrt{\BH} \Bigg( \sum_{k=0}^{\ell} \tfrac{(2\mHC-2\ell+1)!}{(k+2\mHC-2\ell+1)!} \tbinom{k+\mHC-\ell}{\mHC-\ell} b_{\ell-k} (-1)^{\ell-k} \Bigg) (-B)^{\mHC} \bx_0 \Bigg\|_2^2
+\cO(\tau^{2\mHC+1}) .
\end{split}
\end{equation*}
Using $b_p$, $p=0,\ldots,\ell-1$ from the previous steps, and choosing 
\begin{equation} \label{b_ell}
 b_\ell 
:=-(-1)^\ell \sum_{k=1}^{\ell} \tfrac{(2\mHC-2\ell+1)!}{(k+2\mHC-2\ell+1)!} \tbinom{k+\mHC-\ell}{\mHC-\ell} b_{\ell-k} (-1)^{\ell-k}
\end{equation}
yields $S_{\mHC-\ell}=\cO(\tau^{2\mHC+1})$.
\medskip
Choosing these~$b_\ell$, we have verified~\eqref{est:g:mHC:1st}.
Thus, using the ansatz~\eqref{x_ast} for $\bx_\tau$ implies that $g(\bx_\tau;\tau)$ from~\eqref{est:g:mHC} equals the r.h.s.~of~\eqref{est:g:mHC:2nd}.
This proves that the identity~\eqref{SolutionNorm:expansion} holds for $\tau\in[0,1]$.
\end{proof}
}
 
\bigskip
\bigskip

\textbf{Acknowledgment:}
The first author (FA) was supported by the \makeblue{Austrian Science Fund (FWF) via the} FWF-funded SFB \# F65.
The second author (AA) was supported \makeblue{by the Austrian Science Fund (FWF)}, partially via the FWF-doctoral school "Dissipation and dispersion in non-linear partial differential equations'' (\# W1245) and the FWF-funded SFB \# F65.
The third author (EC) was partially supported by U.S. N.S.F. grant DMS 1501007.

The second author (AA) acknowledges fruitful discussions with Laurent Miclo on the possible connection between the hypocoercivity index and the short time decay rate, which eventually led to Theorem~\ref{th:HC-decay}. \makeblue{The authors are grateful to the anonymous referee, whose suggestions helped to simplify the proof of Theorem~\ref{th:HC-decay}. }



\end{document}